%% file: main.tex
\documentclass[12pt]{amsart}

\input{Paper_Style.sty}

\usepackage{tikz}
\usepackage{booktabs}

\input{1_Alphabet_Commands.tex}

\input{2_Math_Commands.tex}

\newcommand{\dist}[1]{d_{#1}}
\newcommand{\shift}{\tau}
\newcommand{\alphabet}{\Sigma}
\DeclareMathOperator{\edit}{\mathsf{edit}}
\DeclareMathOperator{\stutter}{\mathsf{stut}}
\newcommand{\ngram}[2]{[{#1}]_{#2}}
\newcommand{\multiplicity}[2]{c(#1, #2)}

\begin{document}
\title{A weighted angle distance on strings}
\author{Grant Molnar}\email{gmolnar@arka.org}\thanks{This work was carried out at ARKA. Patent pending.}

\begin{abstract}
	We define a multi-scale metric $d_\rho$ on strings by aggregating angle distances between all $n$-gram count vectors with exponential weights $\rho^n$. We benchmark $d_\rho$ in DBSCAN clustering against edit and $n$-gram baselines, give a linear-time suffix-tree algorithm for evaluation, prove metric and stability properties (including robustness under tandem-repeat stutters), and characterize isometries.
\end{abstract}

\maketitle

\section{Introduction}
\label{sec:intro}

This paper was motivated by a simple question: can we define a natural metric on strings for which
\[
(ab)^4 = abababab
\quad\text{and}\quad
(ab)^{20} = abababababababababababababababababababab
\]
are close? 

In many symbolic-sequence settings---including text processing, computational biology,
and log or trace analysis---repeated use of a short motif is often semantically mild:
the string ``stutters'' without changing its essential type. Classical edit distances (e.g.\ Levenshtein or Damerau--Levenshtein) treat this as a large change, since inserting additional tandem copies requires many operations, and the distance grows linearly with the number of repeated blocks.
When repetition scale is a nuisance factor rather than a signal, this behavior is undesirable.

Rather than introducing new edit operations, we take a feature-based, multi-scale approach. Fix a finite alphabet $\alphabet$.
For each string $S\in\alphabet^*$ and each length $n\ge 1$ we form the $n$-gram count vector $\ngram{S}{n}\in\bbR^{\alphabet^n}$. We compare these vectors using the angle distance~$\theta$ (the arccosine of cosine similarity), and we aggregate the discrepancy across all $n$ using an exponentially decaying weight~$\rho^n$:
\begin{equation}\label{eq:intro:dist-rho}
    \dist{\rho}(S,T)\coloneqq \sum_{n\ge 1} \rho^n\,\theta_n(S,T),
    \qquad
    \theta_n(S,T)\coloneqq \theta(\ngram{S}{n},\ngram{T}{n}).
\end{equation}
Intuitively, $\theta_n(S,T)$ measures how similar the local $n$-gram statistics of $S$ and~$T$ are,
and the parameter $\rho$ controls how strongly we discount longer-range interactions.
Small $\rho$ emphasizes short grams (local structure), while larger $\rho$ assigns more weight to
longer grams (global structure).
Although \eqref{eq:intro:dist-rho} is defined for all $\rho>0$ (only finitely many terms are nonzero),
when $0<\rho<1$ the series viewpoint is most transparent and yields uniform bounds such as
\[
    \dist{\rho}(S,T) \le \frac{\pi}{2}\cdot\frac{\rho}{1-\rho}
    \qquad
    \text{for all } S,T\in\alphabet^*.
\]
In particular, $\dist{\rho}$ is insensitive to length at large scales: it is designed to compare
pattern profiles rather than raw counts.

\subsection*{Contributions}

We now summarize the main results of the paper.

\begin{itemize}
    \item \textbf{A multi-scale cosine geometry on strings.}
    For every $\rho>0$ we define the weighted angle distance $\dist{\rho}$ on $\alphabet^*$
    and show that it is a metric (\Cref{Proposition: dist rho is a metric again}).
    The construction is alphabet-invariant: any permutation of~$\alphabet$ induces an isometry of
    $(\alphabet^*,\dist{\rho})$.

    \item \textbf{Linear-time computation.}
    We give a generalized suffix-tree (or suffix-array) algorithm that computes $\dist{\rho}(S,T)$
    in time $O(\abs{S}+\abs{T})$
    (\Cref{Proposition: computation correctness and complexity}), by reducing the necessary
    dot products and norms across all $n$ as aggregates over internal nodes of the generalized
    suffix tree.

    \item \textbf{Empirical evaluation.}
    In \Cref{Section: Experiments} we evaluate $\dist{\rho}$ inside a fixed DBSCAN clustering
    pipeline on three labeled sequence benchmarks, comparing against edit distances and
    fixed-scale $k$-gram baselines.
    The results suggest that weighted angle distances can be particularly effective on datasets
    whose class structure is driven by tandem-repeat statistics, while remaining competitive with
    classical baselines on more conventional corpora.

    \item \textbf{Quantitative control of local edits and repetition.}
    We prove explicit Lipschitz-type bounds for elementary edits
    (\Cref{thm-insertions-deletions-substitutions}) and for stutter operations
    (\Cref{Theorem: stutter}), formalizing the idea that repeating a block in place
    produces only a controlled perturbation in $\dist{\rho}$.

    \item \textbf{Topological and geometric structure.}
    We show that $\dist{\rho}$ induces the discrete topology on $\alphabet^*$, but that when
    $0<\rho<1$ the metric completion is nontrivial.
    We identify this completion with the disjoint union of finite strings and shift-invariant
    Borel probability measures on one-sided infinite strings, equipped with a natural extension
    of~$\dist{\rho}$ (\Cref{Theorem: Completion of alphabet^*}).
    This connects $\dist{\rho}$ to the geometry of empirical and limiting $n$-gram distributions.

    \item \textbf{Rigidity of symmetries.}
    We completely characterize the isometry group of $(\alphabet^*,\dist{\rho})$:
    every isometry is either a symbol permutation or the composition of a symbol permutation with
    string reversal (\Cref{Theorem: isometries of dist rho}).
\end{itemize}

\subsection*{Organization}

In \Cref{Section: Preliminaries and Definitions} we fix notation and introduce $n$-gram count vectors
and the angle distance on nonnegative vectors.
\Cref{Section: Related Work} places the construction in context with standard edit and $n$-gram
distances.
\Cref{Section: Computation} gives the linear-time computation via generalized suffix trees.
\Cref{Section: Experiments} reports clustering experiments and discusses the dependence on~$\rho$.
Finally, \Cref{Section: Conclusion} summarizes takeaways and directions for future work.

\section{Preliminaries and Definitions}\label{Section: Preliminaries and Definitions}

We begin by recalling some standard notation from formal language theory.

Throughout, we let $\alphabet$ denote a fixed alphabet with finitely many characters, we write $\alphabet^n$ for the length $n$ strings over the alphabet $\alphabet$, we write $\alphabet^{\leq N} \coloneqq \bigsqcup_{n \leq N} \alphabet^n$ for the set of strings of length at most $N$ over $\alphabet$, and we write $\alphabet^* \coloneqq \bigsqcup_{n \geq 0} \alphabet^n$ for the set of all strings over $\alphabet$. If $S \in \alphabet^m$, we write $\abs{S} = m$. For $S_1, S_2 \in \alphabet^*$, we write $S_1 S_2$ for the concatenation of $S_1$ and $S_2$. The map $\abs{\cdot} : \alphabet^* \to \bbZ_{\geq 0}$ is a morphism of monoids, so $\abs{S_1 S_2} = \abs{S_1} + \abs{S_2}$. We write $\epsilon$ for the empty string, so for any $S \in \alphabet^*$ we have $\epsilon S = S \epsilon = S$. 

\begin{definition}\label{Definition: substrings}
    We say $Q$ is a \defi{substring} of $S$ if $S = P_1 Q P_2$ for some $P_1, P_2 \in \alphabet^*$. We say $Q$ has \defi{multiplicity $\multiplicity S Q$ in $S$} if there are exactly $\multiplicity S Q$ pairs $(P_1, P_2)$ such that $S = P_1 Q P_2$. If $Q$ is a substring of $S$ and $\abs Q = n$, we say that $Q$ is an \defi{$n$-gram} of $S$.
\end{definition}

\begin{remark}
    In formal language theory, a string is commonly referred to as a ``word''; in genomics, a $k$-gram is typically referred to as a ``$k$-mer''. We opt to retain the phraseology more common in computer science and machine learning.
\end{remark}

We write $\bbN = \set{1, 2, \dots}$ for the set of positive integers. We identify $\bbR^{\alphabet^n}$ with the space of real-valued functions on $\alphabet^n$, so that a vector $v \in \bbR^{\alphabet^n}$ is written as $v : X \mapsto v(X)$.

\begin{definition}
    For $S \in \alphabet^*$ and $n \in \bbN$, we define $\ngram S n \in \bbR^{\alphabet^n}$ as follows:
    \[
        \ngram S n : X \mapsto \multiplicity S X. 
    \]
\end{definition}

Thus, $\ngram S n$ counts all length $n$ substrings of $S$.

\begin{example}
    Let $\alphabet = \set{a, b}$ and let $S = abab$. Then $\ngram S 2 = \set{aa : 0, \ ab : 2, \ ba : 1, \ bb : 0}$. Omitting 0 entries, we write $\ngram S 2 = \set{ab : 2, \ ba : 1}$.
\end{example}

Recall that for $v \in \bbR^N$ a vector, we write $\norm{v} = \sqrt{v \cdot v}$ for the Euclidean norm of $v$.

\begin{definition}\label{Definition: theta(u, v)}
    For any integer $N \geq 0$, we define the \defi{angle distance} $\theta : \bbR_{\geq 0}^N \times \bbR_{\geq 0}^N \to [0, \pi/2]$ by
    \[
    \theta(u, v) \coloneqq \begin{cases}
        \arccos\parent{\frac{u \cdot v}{\norm{u} \norm{v}}} & \ \text{if} \ u, v \neq 0, \\
        0 & \ \text{if} \ u = v = 0, \\
        \frac{\pi}{2} & \ \text{else}.
    \end{cases}
    \]
\end{definition}

The map $\theta$ is a pseudometric on its domain, and restricts to a metric on the unit sphere in $\bbR_{\geq 0}^M$.

We are now ready to define our main object of study.

\begin{definition}\label{Definition: weighted angle distance}
    Let $\alphabet$ be a finite alphabet, and let $\rho > 0$ be arbitrary. We define the \defi{$\rho$-weighted angle distance} $\dist \rho : \alphabet^* \times \alphabet^* \to \bbR_{\geq 0}$ as
    \[
    \dist \rho : (S, T) \mapsto \sum_{n \geq 1} \rho^n \theta_n(S, T),
    \]
    where
    \[
    \theta_n(S, T) \coloneqq \theta(\ngram S n, \ngram T n).
    \]
\end{definition}

Note that $\theta(\ngram S n, \ngram T n) = 0$ whenever $n > \max\parent{\abs{S}, \abs{T}}$, so $\dist \rho$ is well-defined even when $\rho \geq 1$. 

In essence, for each scale $n$, we embed a string $S$ as its $n$-gram count vector $\ngram S n$, and compare two strings by the angle between these vectors (so we compare patterns, not lengths). Then $\dist \rho (S, T)$ is given by summing these angular discrepancies across all scales, with an exponential weight $\rho^n$ that lets you tune how much you care about longer vs. shorter structure.

\begin{example}
Let $\alphabet=\{a,b\}$, take $S=ab$ and $T=ba$, and fix $\rho>0$.
At scale $n=1$, both strings have the same unigram counts, so
\[
\ngram S 1 = \{a:1,b:1\}=\ngram T 1
\]
and thus $\theta_1(S,T)=0$. At scale $n=2$, the bigram vectors are supported on different coordinates,
\[
\ngram S 2=\{ab:1\},\qquad \ngram T 2 =\{ba:1\},
\]
so $\ngram S 2$ and $\ngram T 2$ are orthogonal, and $\theta_2(S,T)=\pi/2$.  All higher $n$ contribute $0$, hence
\[
\dist \rho(S,T)=\rho^2\cdot\frac{\pi}{2}.
\]
\end{example}

\begin{prop}\label{Proposition: dist rho is a metric}
    For any $\rho > 0$, the map $\dist \rho : \alphabet^* \times \alphabet^* \to \bbR_{\geq 0}$ is a metric.
\end{prop}

See \Cref{Proposition: dist rho is a metric again} below for a proof of this fact.

The weighted angle distance $\dist \rho$ is built from a standard representation of strings: for each $n\ge1$, a string $S\in\alphabet^*$ determines an $n$-gram count vector $\ngram{S}{n}\in\bbR^{\alphabet^n}$ encoding its pattern frequencies at scale $n$.

Using the angle (cosine) distance between these vectors is natural because it compares profiles rather than magnitudes, so it is comparatively insensitive to length and to tandem repetition (stutter) while still detecting changes in local structure.

Summing these angular discrepancies over all $n$ yields a genuinely multi-scale notion of similarity, with the weight $\rho^n$ providing a transparent control of how strongly longer-range patterns are emphasized. Moreover, the construction is alphabet-agnostic and admits a linear-time evaluation via generalized suffix trees (or suffix arrays), making it a computationally tractable alternative to edit-based distances when repeated motifs are the dominant signal.

\section{Related Work}\label{Section: Related Work}

The study of distances on strings has a long history spanning theoretical computer science,
information retrieval, and computational biology.  Our $\rho$-weighted angle distance sits at the
intersection of two major themes: $n$-gram feature-space methods and string kernels, and alignment-free approaches for biological sequences.

We begin this section, however, with  the standard approach to string distances.

\subsection{Edit-based distances and stutter resistance}

The most classical and widely used metrics on strings are edit-based.  Hamming distance measures
the number of substitutions between equal-length strings \cite{Hamming1950}.  The Levenshtein
(edit) distance extends this to insertions and deletions, and can be computed in quadratic time
via the dynamic programming algorithm of Wagner and Fischer
\cite{Levenshtein1965,WagnerFischer1974}.  The Damerau--Levenshtein distance further allows
transpositions of adjacent characters, capturing common typo patterns
\cite{Damerau1964}.  These metrics are fundamental in coding theory, pattern matching, spell
correction, and sequence alignment, precisely because they track local, position-sensitive edits.

However, edit distances can behave poorly when the main notion of similarity is repetition
scale rather than literal edit alignment.  For example, $(ab)^4$ and $(ab)^{20}$ are far apart in
Levenshtein distance despite having essentially the same local $n$-gram structure repeated at
different scales.  In such settings one would like a metric that treats stutters
(repetitive block expansions) as relatively small perturbations.

Recent work by Petty et al.\ proposes a new edit-based distance designed to address some of these
weaknesses, with a particular focus on capturing repetition-robust similarity while preserving an
operational ``editing'' interpretation \cite{Petty2022}.  Their construction augments the edit
operations with a family of stutter-like transformations, but requires one to commit in advance to
a menu of permissible stutters and to perform substantial precomputation to support them.

Our approach is complementary: rather than enriching the suite of permissible edits, we adopt a
multi-scale, compositional viewpoint based on weighted angular comparisons of $n$-gram count
vectors.  This shifts the inductive bias from ``minimal edits'' to ``similar local statistics
across scales,'' which naturally compresses length inflation caused by repeated blocks.

\subsection{\texorpdfstring{$n$-gram}{n-gram} feature spaces, cosine geometry, and string kernels}

A large body of work represents strings by their substring statistics and compares the resulting
vectors in a feature space.  Character $n$-gram profiles have long been successful in tasks such
as language identification and text categorization, especially for short or noisy documents
\cite{CavnarTrenkle1994}.  In classical information retrieval, the vector space model represents
documents and queries as weighted term vectors and uses cosine similarity for comparison
\cite{SaltonWongYang1975,ManningRaghavanSchutze2008}.  This paradigm underlies many modern
practices, including readily available string-similarity toolkits that expose cosine distance on
$n$-gram count vectors as a basic primitive \cite{strsimpy}.

In theoretical and applied work on string kernels, this perspective is formalized in a
reproducing-kernel-Hilbert-space framework.  The spectrum kernel is essentially the inner product
on $k$-gram count vectors, while mismatch-style kernels allow controlled perturbations of
substrings; both families have been highly influential in text and biological sequence
classification \cite{Lodhi2002,Leslie2002Spectrum,Leslie2004Mismatch,Sonnenburg2007}.  From this
viewpoint, one typically fixes a scale $k$ and measures similarity in the associated feature
space.  Related ideas appear in approximate string matching and filtering using $q$-grams
\cite{Ukkonen1992}.

Our construction differs structurally from most fixed-$k$ approaches.  The distance $\dist{\rho}$
incorporates \emph{all} $n$-gram scales via an exponentially decaying weight $\rho^n$ and uses the
angle between raw count vectors at each scale.  This produces a bona fide metric on $\alphabet^*$
(for any $\rho>0$, and uniformly bounded when $\rho\in(0,1)$), while retaining the interpretability
and conditioning benefits associated with cosine geometry.

\subsection{Alignment-free sequence comparison and sketching}

In genomics and metagenomics, alignment-free methods compare biological sequences using $k$-mer
count or frequency vectors, together with various geometric or information-theoretic dissimilarities
\cite{OtuSayood2003,Zielezinski2019}.  Such methods are particularly attractive when alignment is
computationally expensive or conceptually inappropriate, and they often rely on $k$-mer statistics
that are closely analogous to our $n$-gram counts.

Scaling these ideas to very large datasets has led to sketch-based techniques.  For example, Mash
uses MinHash sketches of $k$-mer sets to approximate distances between genomes and metagenomes
extremely quickly, while still maintaining good empirical accuracy \cite{Ondov2016}.  These
approaches demonstrate that compositional statistics can serve as robust fingerprints when explicit
alignment is infeasible or undesirable.

Our metric shares with alignment-free methods the use of substring statistics as primary features,
but differs in several respects: it aggregates all scales $n$ with a principled weighting,
is explicitly cosine-derived at each scale, and is provably a metric on the space of finite
strings.  The suffix-tree-based algorithm we describe in \Cref{Section: Computation} ensures that
this multi-scale construction remains computationally tractable in practice.

\subsection{Compression-based distances}

Another strand of work compares strings via compression, using compressor behavior as a proxy for
shared structure.  The normalized compression distance (NCD) of Cilibrasi and Vit{\'a}nyi
\cite{CilibrasiVitanyi2005} is a prominent example: it defines a universal, task-agnostic
similarity measure derived from the Kolmogorov complexity framework, instantiated with real-world
compressors.  NCD and related methods can capture very rich regularities without hand-crafted
features, but are often sensitive to the choice of compressor and can be less transparent
analytically.

By contrast, our approach affords explicit combinatorial and geometric control over substring
statistics: the metric is built from concrete $n$-gram counts and the geometry of angles in
$\bbR^{\alphabet^n}$, and admits direct analysis of its stability under edits and stutters.

\subsection{Summary}

In summary, $\dist{\rho}$ may be viewed as a metric-level synthesis of ideas that have appeared
independently across $n$-gram feature comparisons
\cite{CavnarTrenkle1994,SaltonWongYang1975,ManningRaghavanSchutze2008},
string kernels \cite{Lodhi2002,Leslie2002Spectrum,Leslie2004Mismatch,Sonnenburg2007},
and alignment-free sequence analysis
\cite{OtuSayood2003,Ondov2016,Zielezinski2019}, while remaining conceptually distinct from both
classical edit distances \cite{Hamming1950,Levenshtein1965,WagnerFischer1974,Damerau1964,Petty2022}
and compression-based similarities \cite{CilibrasiVitanyi2005}.  To the best of our knowledge, none
of the existing constructions simultaneously offers (i) robustness to stutter-like repetition,
(ii) multi-scale $n$-gram sensitivity, and (iii) provable metric structure together with a
linear-time algorithm for exact evaluation on finite strings.

\section{Computation}\label{Section: Computation}

In this section, we describe an efficient method for computing $\dist \rho(S,T)$ for two strings
$S,T\in\alphabet^*$. Throughout, we write $m\coloneqq \abs{S}$, $n\coloneqq \abs{T}$, and $L_{\max} \coloneqq \max(m,n)$.

\subsection{Reducing \texorpdfstring{$\dist \rho$}{d\_rho} to aggregated \texorpdfstring{$n$-gram}{n-gram} statistics}

Recall that
\[
    \dist \rho(S,T)=\sum_{n\ge 1}\rho^n\theta_n(S,T),
\]
where $\theta_n(S,T)=\theta(\ngram S n,\ngram T n)$. For each $n\ge 1$, define the aggregated quantities
\begin{align}
    A_n(S) &\coloneqq \norm{\ngram S n}^2
    = \sum_{W\in\alphabet^n} \multiplicity S W^2,  \label{Equation: An}\\
    A_n(T) &\coloneqq \norm{\ngram T n}^2
    = \sum_{W\in\alphabet^n} \multiplicity T W^2, \nonumber\\
    B_n(S,T) &\coloneqq \ngram S n\cdot \ngram T n
    = \sum_{W\in\alphabet^n} \multiplicity S W \multiplicity T W. \label{Equation: Bn}
\end{align}
Then for $n\le L_{\max} $ we have
\begin{equation}\label{Equation: theta via A and B}
    \theta_n(S,T)=
    \begin{cases}
        \arccos\parent{\dfrac{B_n(S,T)}{\sqrt{A_n(S)A_n(T)}}},
        & A_n(S)A_n(T)\neq 0,\\[10pt]
        0, & A_n(S)=A_n(T)=0,\\[4pt]
        \dfrac{\pi}{2}, & \text{otherwise}.
    \end{cases}
\end{equation}
Moreover, $\theta_n(S,T)=0$ for $n>L_{\max}$. Thus computing $\dist \rho(S,T)$ reduces to computing the sequences $\bigl(A_n(S)\bigr)_{1\le n\le L_{\max} }$, $\bigl(A_n(T)\bigr)_{1\le n\le L_{\max} }$, and $\bigl(B_n(S,T)\bigr)_{1\le n\le L_{\max} }$.

\subsection{Generalized suffix trees}

We now explain how to compute the quantities in
\eqref{Equation: An}--\eqref{Equation: Bn} in time linear in $m+n$,
using a generalized suffix tree (or, equivalently, a suffix array with LCP-interval processing).

Let $\#,\$ \notin \alphabet$ be two fresh terminal symbols,
and set
\[
    U \coloneqq S\#T\$.
\]
A \defi{generalized suffix tree} for $S$ and $T$ is a suffix tree for $U$ in which
each leaf is labeled according to whether its suffix begins in the $S$-block or the $T$-block.
We write $\mathrm{depth}(v)$ for the string depth of a node $v$, i.e.\ the length of
the path label from the root to $v$.

For a node $v$, define
\begin{align*}
    \mathrm{occ}_S(v) &\coloneqq \#\{\text{leaves in the subtree of $v$ whose suffix begins in $S$}\},\\
    \mathrm{occ}_T(v) &\coloneqq \#\{\text{leaves in the subtree of $v$ whose suffix begins in $T$}\}.
\end{align*}
By standard suffix-tree properties, if a string $W\in\alphabet^+$ labels the path to a point on the edge into
$v$ (equivalently, $W$ contains neither $\#$ nor $\$$), then $\multiplicity{S}{W}=\mathrm{occ}_S(v)$ and $\multiplicity{T}{W}=\mathrm{occ}_T(v)$. In other words, along a single edge, all sentinel-free substrings represented by points on that edge have the same occurrence counts in $S$ and in $T$.

This observation yields a compact formula for the aggregated $n$-gram statistics.

\begin{prop}\label{Proposition: A_n and B_n from GST}
Let $S,T\in\alphabet^*$, let $U=S\#T\$$, and let $\mathcal T$ be the generalized suffix tree of $U$.
For each non-root node $v$ of $\mathcal T$, let $\ell_v^-\coloneqq \mathrm{depth}(\mathrm{parent}(v))+1$.

Let $\ell_v^+$ denote the largest integer $d\le \mathrm{depth}(v)$ such that the length-$d$ prefix of the
path label to $v$ lies in $\alphabet^d$ (i.e.\ contains neither $\#$ nor $\$$).
(If the path label to $v$ contains a sentinel, then $\ell_v^+<\mathrm{depth}(v)$.)

Then for every integer $n\ge 1$,
\begin{align*}
A_n(S) &= \sum_{v\neq \mathrm{root}}
\mathbf{1}_{\ell_v^- \le n\le \ell_v^+}\, \mathrm{occ}_S(v)^2,\\
A_n(T) &= \sum_{v\neq \mathrm{root}}
\mathbf{1}_{\ell_v^- \le n\le \ell_v^+}\, \mathrm{occ}_T(v)^2,\\
B_n(S,T) &= \sum_{v\neq \mathrm{root}}
\mathbf{1}_{\ell_v^- \le n\le \ell_v^+}\, \mathrm{occ}_S(v)\mathrm{occ}_T(v).
\end{align*}
Nodes with $\ell_v^+<\ell_v^-$ contribute nothing.
\end{prop}

\begin{proof}
    The distinct substrings of $U$ are in bijection with the integer depths of the non-root points
    on edges of the suffix tree. For a fixed non-root node $v$, the edge from $\mathrm{parent}(v)$ to $v$ contributes exactly one distinct substring of length $n$ for each integer $n$ with $\ell_v^- \le n\le \ell_v^+$, namely the prefix of the edge label of length $n-\mathrm{depth}(\mathrm{parent}(v))$ appended to the path label of $\mathrm{parent}(v)$. Every such substring has multiplicity $\mathrm{occ}_S(v)$ in $S$ and $\mathrm{occ}_T(v)$ in $T$. Summing $\multiplicity S W^2$, $\multiplicity T W^2$, and $\multiplicity S W\,\multiplicity T W$ over all distinct $W\in\alphabet^n$ gives the stated formulas.
\end{proof}

\subsection{A linear-time algorithm}

\Cref{Proposition: A_n and B_n from GST} suggests a simple range-add implementation.

We are ultimately interested only in substrings over $\alphabet$. Accordingly, when we speak of a node $v$ and the substrings it represents, we implicitly restrict our attention to those nodes for which every label on the root-$v$ path lies in $\alphabet$ (i.e. no symbol $\#$ or $\$$ appears). Nodes whose path label crosses a sentinel will be ignored in the sums below.

If either $S$ or $T$ is empty, then $\dist \rho (S, T)$ may be computed directly from $\abs{S}$ and $\abs{T}$ using \Cref{Lemma: We can recover |S| from dist rho} below. Throughout the remainder of this section, we assume $S$ and $T$ are both nonempty.

\begin{alg}
\begin{enumerate}
    \item Build the generalized suffix tree $\mathcal T$ of $U=S\#T\$$.
    \item Compute $\mathrm{occ}_S(v)$ and $\mathrm{occ}_T(v)$ for all nodes $v$ by a single post-order traversal: at a leaf, set $(\mathrm{occ}_S,\mathrm{occ}_T)=(1,0)$ or $(0,1)$ depending on its origin, and at an internal node sum the counts of its children.
    \item Initialize three difference arrays of length $L_{\max} +2$:
          \[
              D_S,\ D_T,\ D_{ST}\in\bbR^{L_{\max} +2}.
          \]
          For each non-root node $v$ whose path label lies in $\alphabet^*$, compute
          \[
              a_v\coloneqq \mathrm{occ}_S(v)^2,\qquad
              b_v\coloneqq \mathrm{occ}_T(v)^2,\qquad
              c_v\coloneqq \mathrm{occ}_S(v)\mathrm{occ}_T(v),
          \]
          and perform the range additions
          \begin{align*}
              D_S[\ell_v^-] &\mathrel{+}= a_v, & D_S[\ell_v^+ + 1] &\mathrel{-}= a_v,\\
              D_T[\ell_v^-] &\mathrel{+}= b_v, & D_T[\ell_v^+ + 1] &\mathrel{-}= b_v,\\
              D_{ST}[\ell_v^-] &\mathrel{+}= c_v, & D_{ST}[\ell_v^+ + 1] &\mathrel{-}= c_v,
          \end{align*}
          ignoring any indices greater than $L_{\max} +1$.
    \item Convert the difference arrays to prefix sums to obtain
          \[
              A_n(S),\ A_n(T),\ B_n(S,T)\qquad (1\le n\le L_{\max} ).
          \]
    \item Compute $\theta_n(S,T)$ from \eqref{Equation: theta via A and B} for $1\le n\le L_{\max} $
          and finally evaluate
          \[
              \dist \rho(S,T)=\sum_{n=1}^{L_{\max} }\rho^n\theta_n(S,T).
          \]
\end{enumerate}
\end{alg}

\begin{prop}\label{Proposition: computation correctness and complexity}
    The algorithm above computes $\dist \rho(S,T)$ correctly.
    Moreover, assuming a linear-time suffix-tree (or suffix-array) construction,
    the total running time is $O(m+n)$ and the memory usage is $O(m+n)$.
\end{prop}

\begin{proof}
    Correctness follows from \Cref{Proposition: A_n and B_n from GST} and
    \eqref{Equation: theta via A and B}.
    The post-order computation of $(\mathrm{occ}_S,\mathrm{occ}_T)$ is linear in the size of $\mathcal T$,
    and the range-add pass performs $O(1)$ work per node.
    Since a generalized suffix tree for two strings has size $O(m+n)$, the stated bounds follow.
\end{proof}

\begin{remark}
    The algorithm above may be implemented purely with suffix arrays. One builds the suffix array of $U=S\#T\$$ and its LCP array, then constructs the LCP-interval tree, which is isomorphic (as a compact trie of all suffixes) to the internal-node structure of the suffix tree. The leaf-origin labels and the post-order aggregation of $(\mathrm{occ}_S,\mathrm{occ}_T)$ are then carried out on this interval tree, and the same range-add formulas apply.
\end{remark}

The same approach yields a fast method for computing $\dist \rho(S,\epsilon)$ and for computing all $\theta_k(S,T)$ simultaneously. In particular, the computation time does not depend on $\abs{\alphabet}^n$, avoiding the exponential blow-up inherent in explicit $n$-gram vector construction.

\section{Experiments}\label{Section: Experiments}

%https://archive.ics.uci.edu/dataset/69/molecular+biology+splice+junction+gene+sequences
%https://ftp.ebi.ac.uk/pub/databases/Pfam/current_release/
%https://ml.jku.at/software/LSTM_protein/

This section evaluates the $\rho$-weighted angle distance $\dist{\rho}$ as a
string metric for \emph{stutter-heavy} data, i.e.\ datasets where examples often
differ by tandem-repeat expansions together with mild local variation.  Our goal
is not to ``tune for wins'' on a carefully curated benchmark, but to probe the
behavior of $\dist{\rho}$ under a fixed and label-blind clustering protocol, and
to compare it against classical edit distances and common $k$-gram baselines.

\subsection{Experimental philosophy}\label{subsec:exp_philosophy}

Two design constraints guided these experiments.

\begin{itemize}
\item \textbf{No post-hoc dataset curation.}
We intentionally did \emph{not} curate datasets to make $\dist{\rho}$ look good. We focused on stutter-heavy datasets because stutter-resistance is the intended use-case, but within that scope we report results for every labeled dataset we downloaded and processed during this study.
\item \textbf{Methodology fixed before runs.}
The methodology described in \Cref{subsec:methodology} was chosen before any experiments were run and then applied uniformly to all datasets and all metrics. We did not hand-adjust DBSCAN parameters, random seeds, or sampling settings.  
\end{itemize}

\subsection{Datasets}
We evaluate clustering quality on three labeled datasets of strings drawn from public genomics repositories.  Each dataset is normalized into a uniform three-column schema
\[
(\texttt{label},\ \texttt{sample\_id},\ \texttt{sequence}),
\]
so that the clustering pipeline is agnostic to the data source and only depends on the string distance.

\begin{itemize}
\item \textbf{UCI splice-junction gene sequences (\texttt{splice}).}
This dataset is the UCI Machine Learning Repository ``Molecular Biology (Splice-junction Gene Sequences)'' corpus.  Each sample is a fixed-length (60) DNA-like string with ambiguity codes, labeled by splice-site category (\texttt{EI}, \texttt{IE}, or \texttt{N}).  We preserve the original instance identifier as \texttt{sample\_id} and export the normalized triples.

\item \textbf{UCSC hg19 tandem repeats (\texttt{ucsc\_trf}).}
We extract tandem-repeat records from the UCSC Genome Browser database table \texttt{simpleRepeat} (hg19), which is derived from Tandem Repeats Finder (TRF) calls.  Each record supplies a repeated ``unit'' (motif) along with its genomic interval.  We treat the repeat unit as the class label and subsample motif classes to obtain a balanced labeled dataset; \texttt{sample\_id} encodes the genomic interval and metadata, and \texttt{sequence} is the repeated string.

\item \textbf{NCBI STRSeq locus collections (\texttt{strseq}).}
We aggregate multiple STRSeq BioProjects from NCBI and download the associated FASTA records through the Entrez E-utilities API.  Each record is labeled by locus (e.g., \texttt{D8S1179}, \texttt{vWA}, \texttt{FGA}, etc.), yielding a multi-class dataset with substantial tandem-repeat structure.
\end{itemize}

\subsection{Distances compared}\label{subsec:distances_actual}

We compare $\dist{\rho}$ to standard edit-based and $k$-gram baselines.  All
distances are computed on raw strings.

\begin{itemize}
\item \textbf{Weighted angle distance ($\dist \rho$).}
We evaluate $\dist{\rho}$ for $\rho \in \{0.1,0.2,\dots,1.0\}$, truncating the
sum at \texttt{max\_n=60} (which is exact whenever both strings have length
$\le 60$ and otherwise is a controlled approximation for $\rho < 1$).

\item \textbf{Fixed-scale $k$-gram cosine angle.}
For $k\in\{3,4,5,6\}$ we compute the cosine angle between the $k$-gram count
vectors (this is exactly the $\theta_k$ term of our construction).

\item \textbf{$k$-gram Jensen--Shannon distance.}
For $k\in\{3,4,5,6\}$ we compute the Jensen--Shannon distance between empirical
$k$-gram distributions.

\item \textbf{Edit-distance baselines.}
We include Levenshtein distance, Damerau--Levenshtein distance, and the LCS
distance (\texttt{rapidfuzz.distance.LCSseq}).
\end{itemize}

\subsection{Methodology}\label{subsec:methodology}

For each dataset and each distance $d(\cdot,\cdot)$ we compute the full
pairwise distance matrix
\[
D_{ij} \coloneqq d(x_i,x_j)
\qquad (1\le i,j\le N),
\]
and run DBSCAN with \texttt{metric="precomputed"}.

We tune \texttt{eps} and \texttt{min\_samples} using Optuna (TPE sampler) for
\texttt{n\_trials=100}.  The objective is the silhouette score computed on non-noise points (DBSCAN noise label \texttt{-1} is removed before silhouette is evaluated).  If DBSCAN produces fewer than two non-noise clusters, the silhouette is defined to be $-1$ to prevent degenerate solutions from being selected.

\begin{remark}
    We made the decision to tune to the silhouette score because it is available without labels, is inexpensive once $D$ is computed, and it provides a consistent tuning criterion across all datasets and metrics.  We emphasize that silhouette is \emph{not} assumed to correlate perfectly with ground-truth clustering quality; it is used only as a practical, label-free model-selection heuristic.
\end{remark}

To avoid dataset-specific hand-tuning of the \texttt{eps} scale, we restrict the
search range using quantiles of the observed pairwise distance distribution.
Concretely, let $\mathcal{D}$ be the multiset of upper-triangular entries of $D$.
We set
\[
\texttt{eps} \in
\Bigl[\operatorname{Quantile}_{0.02}(\mathcal{D}),\ \operatorname{Quantile}_{0.20}(\mathcal{D})\Bigr],
\]
with standard numeric guards when the interval degenerates.  We search
\[
\texttt{min\_samples} \in \{3,5,8,13\}.
\]
The selected hyperparameters are the Optuna best trial parameters, and the
final DBSCAN labeling is obtained by rerunning DBSCAN with those parameters.

After selecting DBSCAN parameters using silhouette alone, we evaluate the
resulting clustering against the provided dataset labels using:
\begin{itemize}
\item \textbf{Adjusted Rand Index (ARI)} and
\item \textbf{Normalized Mutual Information (NMI).}
\end{itemize}

Noise points (DBSCAN label \texttt{-1}) are treated as an additional cluster
label for ARI/NMI computation so that all sampled points are evaluated under a
single partition.

In addition to ARI/NMI, we record diagnostic quantities including the selected
(\texttt{eps}, \texttt{min\_samples}), the achieved silhouette score (on non-noise
points), the number of non-noise clusters, and the fraction of points labeled
as noise.  All per-run outputs are written to \texttt{results.csv} and
corresponding plots are saved for each (dataset, distance) pair.
The full experimental pipeline---distance implementations, dataset ingestion,
DBSCAN/Optuna harness, and figure generation---is available
at~\cite{WeightedAngleDistanceCode}.

\subsection{Analysis}

\Cref{tab:best_by_family} reports, for each dataset, the best ARI/NMI achieved by each distance family (maximizing over its parameter sweep).

The main qualitative patterns are:

\begin{itemize}
    \item \textbf{\texttt{splice}: no method clusters well.}
    All distances collapse to a single dominant cluster and achieve essentially zero ARI/NMI (best ARI $\approx 0.0003$). See also the flat weighted-angle sweep in \Cref{fig:splice:wa_ari_nmi_rho}. 
    \item \textbf{\texttt{strseq}: edit distances dominate.}
    LCS achieves ARI $=0.8825$ and NMI $=0.8970$, with the other two edit baselines close behind. The weighted angle distance is competitive with the $k$-gram baselines (best at $\rho=0.6$ with ARI $=0.5060$, NMI $=0.7527$), but it does not outperform classical edit distances on this dataset. \Cref{fig:strseq:wa_ari_nmi_rho} shows a clear ``sweet spot'' around $\rho\in[0.5,0.7]$.
    \item \textbf{\texttt{ucsc\_trf}: weighted angle dominates.}
    Here the weighted angle distance is the strongest method overall (best ARI at $\rho=0.6$: ARI $=0.4596$, NMI $=0.8632$), while edit distances fail catastrophically (ARI $\approx 0$, NMI $\approx 0.02$). This is the cleanest evidence in our suite that the weighted angle construction captures repeat-driven structure that is largely invisible to generic edit distances. The $\rho$-sweep in \Cref{fig:ucsctrf:wa_ari_nmi_rho} is stable over a broad plateau, with best NMI attained at $\rho=1.0$.
\end{itemize}

These rankings agree with the per-distance ``best-to-worst'' listings in our run summary.

\begin{remark}
    The experiments above emphasize clustering quality rather than speed. In our current reference implementation~\cite{WeightedAngleDistanceCode}, weighted angle distance computations are substantially slower than the baselines, despite being linear-time per pair in principle: the bottleneck is Python-level overhead combined with the need to compute $O(N^2)$ pairwise distances for DBSCAN, with no caching across parameter sweeps and no parallelism. We therefore report timing only as a coarse reference (\Cref{tab:runtime}) and view performance optimization (vectorization, memoization of repeated trigonometric calls, and parallel distance-matrix construction) as future work.
\end{remark}

\Cref{tab:rho_stability} summarizes how sensitive the weighted angle method is to $\rho$.
Across our three datasets, the choice $\rho=0.6$ maximizes ARI in all cases, while $\rho=1.0$ slightly improves NMI on \texttt{ucsc\_trf}. The variability across $\rho$ is largest on \texttt{strseq} (where the task is ``easier'' for edit distances and weighted angle competes with them), and smallest on \texttt{ucsc\_trf} (where many $\rho$ values fall on the same high-performing plateau).

ARI and NMI are highly consistent as rankings of distances (\Cref{tab:rank_consistency}), with Spearman correlations above $0.93$ for every dataset. By contrast, silhouette has weak (and sometimes negative) rank correlation with ARI/NMI, most notably on \texttt{strseq}. This is an expected consequence of our label-blind DBSCAN tuning: silhouette is used only to choose $(\varepsilon,\texttt{min\_samples})$ within each distance, and it should be interpreted as an internal ``cluster compactness'' diagnostic rather than a proxy for label recovery.
The scatter in \Cref{fig:strseq:sil_vs_ari} illustrates this mismatch concretely.

DBSCAN noise points are a key ``escape valve'' for distances that induce highly non-uniform neighborhoods.
\Cref{fig:ucsctrf:ari_vs_noise} plots ARI against the noise fraction across all tested distances on \texttt{ucsc\_trf}.
On that dataset, high-performing methods achieve good ARI without labeling an extreme fraction of points as noise; on \texttt{strseq} the best weighted-angle setting ($\rho=0.6$) yields a noticeably larger noise fraction than edit distances, suggesting a tradeoff between repeat-sensitivity and conservative cluster formation.

\section{Conclusion}
\label{Section: Conclusion}

We introduced the $\rho$-weighted angle distance $\dist{\rho}$, a multi-scale metric on
$\alphabet^*$ obtained by comparing all $n$-gram count profiles of two strings through
angular discrepancies and aggregating these discrepancies with an exponential weight.
The resulting geometry is ``cosine-like'' at every scale and is deliberately designed to be
insensitive to repetition magnitude: stutters and changes in motif multiplicity are treated as
controlled perturbations rather than as large edit costs.
Along the way we established several structural properties: $\dist{\rho}$ is a genuine metric for
every $\rho>0$, it has a rigid and explicitly described isometry group, and for $0<\rho<1$ it
admits a nontrivial completion that naturally incorporates shift-invariant probability measures on
one-sided infinite strings.
This completion viewpoint clarifies that $\dist{\rho}$ can be interpreted as a metric on empirical
and limiting $n$-gram distributions, not merely as a distance between finite words.

From a computational standpoint, we showed that $\dist{\rho}(S,T)$ can be computed in
$O(\abs{S}+\abs{T})$ time using generalized suffix trees (or suffix arrays), by reducing the
cross-scale dot products and squared norms to aggregates over internal nodes.
This makes $\dist{\rho}$ viable as a primitive for large-scale sequence comparison in principle,
though the end-to-end runtime in our current Python reference implementation remains dominated by
engineering choices in the surrounding pipeline.
In particular, the DBSCAN experiments in \Cref{Section: Experiments} require constructing full
pairwise distance matrices; our implementation does not currently cache repeated computations
across parameter sweeps, vectorize the trigonometric calculations, or exploit parallelism.
Consequently, the reported wall-clock times should be read as a baseline for a straightforward
prototype rather than as an intrinsic limitation of the methodology.

Empirically, the experiments suggest the following picture.
On datasets whose class structure is driven primarily by tandem-repeat statistics,
weighted angle distances can substantially outperform classical edit distances and fixed-scale
$k$-gram baselines, and performance is reasonably stable across a broad range of~$\rho$.
On more conventional sequence corpora, edit distances may remain the best choice, and in some
settings (such as \texttt{splice} in our suite) DBSCAN itself appears to be a poor model of the
ground-truth labeling regardless of the underlying distance.
Overall, the weighted angle construction seems most valuable when repetition scale should be
discounted but repetition \emph{structure} should still matter.

Several directions look immediately promising.
On the algorithmic side, a performant implementation could combine (i) suffix-array based
computation of the required aggregates, (ii) memoization of repeated subexpressions across $\rho$,
and (iii) parallel construction of distance matrices or approximate neighborhood queries.
On the modeling side, one can consider alternative weighting schemes (beyond geometric weights)
to emphasize a chosen range of $n$-gram scales, or learn such weights from unlabeled data.
Finally, the completion theorem hints at a broader viewpoint in which $\dist{\rho}$ is a metric on
objects closer to symbolic dynamical systems than to isolated strings; exploiting this structure
for tasks such as classification, anomaly detection, or retrieval is an appealing target for
future work.

\section{Appendix: Geometry}
\label{Appendix: Geometry}

In this appendix, we conduct a more careful analysis of the mathematical properties of $\dist \rho$. We begin with a proof of \Cref{Proposition: dist rho is a metric}, which we restate here for convenience.

\begin{prop}[\Cref{Proposition: dist rho is a metric}]\label{Proposition: dist rho is a metric again}
    For any $\rho > 0$, the map $\dist \rho : \alphabet^* \times \alphabet^* \to \bbR_{\geq 0}$ is a metric.
\end{prop}

\begin{proof}
    Symmetry and the triangle inequality both inherit from the pseudometricity of $\theta$, as does the triviality $\dist \rho(S, S) = 0$ for each $S \in \alphabet^*$. It suffices to show that if $\dist \rho (S, T) = 0$ then $S = T$.

    Let $S, T \in \alphabet^*$ be arbitrary. Note that $\dist \rho (S, T) = 0$ if and only if $\theta(\ngram S n, \ngram T n) = 0$ for all $n$. If $\abs{S} > \abs{T}$ then letting $m = \abs{S}$ we have $\theta(\ngram S m, \ngram T m) = \pi / 2$ so $\dist \rho (S, T) \neq 0$. Suppose now that $\dist \rho (S, T) = 0$. We must have $\abs S = \abs T = m$ for some $m$. Then $\theta(\ngram S m, \ngram T m) = 0$, but $\ngram S m = \set{S : 1}$ and $\ngram T m = \set{T : 1}$, so $S = T$ as desired.
\end{proof}

\begin{remark}\label{Remark: rho-all}
    The proof of \Cref{Proposition: dist rho is a metric again} in fact shows that
    $\dist\rho$ is a metric on $\alphabet^*$ for every $\rho>0$.
    For each fixed $\rho>0$ the induced topology on $\alphabet^*$ is discrete
    (see \Cref{Proposition: discrete topology on strings}), but when $\rho\in(0,1)$ the metric
    is still topologically interesting (see \Cref{Theorem: Completion of alphabet^*}), and is uniformly bounded by $\frac{\pi}{2}\cdot\frac{\rho}{1-\rho}$. We therefore restrict attention to $\rho\in(0,1)$ in the main text.
\end{remark}

From now on, unless explicitly stated otherwise, we assume $\rho \in (0,1)$.

\subsection{Respecting edits}\label{sec:edits}

In this subsection, we prove that $\dist \rho$ is well-behaved under insertion, deletion, substitution, and stutter.

Our first theorem controls the behavior of $\dist \rho$ under insertions, deletions, and substitutions.

\begin{theorem}\label{thm-insertions-deletions-substitutions}
    Let $\rho \in (0, 1)$. Let $P, Q \in \alphabet^*$ be strings of lengths $m$ and $n$ respectively, and let $a, b \in \alphabet$ be arbitrary. Then
    \begin{equation}\label{Equation: insertion and deletion}
    \dist \rho (P a Q, P Q) \leq \frac{\pi \sqrt 2}{\sqrt {L_{\edit}}} A_{M_{\edit}}(\rho) + \frac{\pi}{2} \cdot \frac{\rho^{{K_{\edit}}+1}}{1 - \rho} + \frac \pi 2 \rho^{{L_{\edit}} + 1}
    \end{equation}
    and
    \begin{equation}\label{Equation: substitution}
    \dist \rho(PaQ, PbQ) \leq \frac{\pi}{\sqrt {L_{\edit}}} A_{M_{\edit}}(\rho) + \frac{\pi}{2} \cdot \frac{\rho^{{K_{\edit}}+1}}{1 - \rho},
    \end{equation}
    where ${L_{\edit}} \coloneqq m + n$, ${K_{\edit}} \coloneqq \floor{{L_{\edit}}/2}$, ${M_{\edit}} \coloneqq \min(m + 1, n + 1)$, and 
    \begin{equation}\label{Equation: AM(rho)}
    A_{M}(\rho) \coloneqq \sum_{k = 1}^\infty \rho^k \min(k, M)
    %= \sum_{1 \leq k \leq M} k\rho^k + \frac{M \rho^{M+1}}{1 - \rho}
    = \frac{\rho\bigl(1-(M+1)\rho^M + M\rho^{M+1}\bigr)}{(1-\rho)^2} + \frac{M \rho^{M+1}}{1 - \rho}.
    \end{equation}
\end{theorem}

\begin{proof}
    If $x, y \in \bbR^N$ are nonzero vectors, then the Cauchy-Schwarz inequality implies
    \begin{equation}\label{Equation: inequality between theta and arcsine}
        \theta(x, y) \leq 2 \arcsin\parent{\min\!\parent{1,\ \frac{\norm{x - y}}{2 \min \parent{\norm x, \norm y}}}}
        \le \pi \cdot \frac{\norm{x-y}}{2\min(\norm x,\norm y)}.
    \end{equation}

    Now let $\rho \in (0, 1)$, let $P, Q \in \alphabet^*$ be strings of lengths $m$ and $n$ respectively, and let $a, b \in \alphabet$ be arbitrary, and let ${L_{\edit}}, {K_{\edit}},$ and ${M_{\edit}}$ be as above.

    We proceed to prove \eqref{Equation: insertion and deletion} by comparing $PQ$ with $P a Q$. For $k \leq {L_{\edit}}$, only $k$-grams that cross the junction between $P$ and $Q$ are removed, and only $k$-grams that include the inserted $a$ are added. We define
    \begin{equation}\label{Definition: ck}
    c_k \coloneqq \max\parent{0,\ \min(m, k - 1) - \max(1, k - n) + 1}
    \end{equation}
    to be the number of $k$-length windows in $P Q$ that cross from $P$ to $Q$, and define
    \begin{equation}\label{Equation: dk}
    d_k \coloneqq \max\parent{0,\ \min(m, k - 1) - \max(0, k - n - 1) + 1}
    \end{equation}
    to be the number of $k$-length windows in $P a Q$ that include the new $a$. If $m, n \geq k$, these equations simplify to $c_k = k - 1$ and $d_k = k$.

    Each of the $c_k$ removed windows decreases one coordinate of $\ngram{PQ}{k}$ by $1$, and each of the $d_k$ new windows increases one coordinate by $1$. By the triangle inequality,
    \begin{equation}\label{Equation: L2 insertion safe}
        \norm{\ngram {PaQ} k - \ngram{PQ} k} \leq c_k + d_k.
    \end{equation}
    Also $\theta_{{L_{\edit}} + 1}(P a Q, P Q) = \frac \pi 2$. Since $\ngram{PQ}k$ has nonnegative integer entries summing to ${L_{\edit}}-k+1$, we have $\norm{\ngram{PQ}k}\ge \sqrt{{L_{\edit}}-k+1}$, and similarly $\norm{\ngram{PaQ}k}\ge \sqrt{{L_{\edit}}-k+1}$. Hence, by \eqref{Equation: inequality between theta and arcsine},
    \[
        \theta_k(PaQ, PQ)
        \le
        \pi\cdot \frac{c_k+d_k}{2\sqrt{{L_{\edit}}-k+1}}
        \qquad (k\le {L_{\edit}}).
    \]
    Therefore
    \begin{equation}\label{Equation: bound on distance from PaQ to PQ}
    \dist \rho (P a Q, P Q)
    \leq 
    \sum_{1 \leq k \leq {L_{\edit}}} \rho^k \cdot \pi \frac{c_k + d_k}{2 \sqrt{{L_{\edit}} - k + 1}} + \frac \pi 2 \rho^{{L_{\edit}} + 1}.
    \end{equation}

    For $k\le {K_{\edit}}$ we have ${L_{\edit}}-k+1\ge {L_{\edit}}/2$, hence
    \[
        \pi \frac{c_k + d_k}{2 \sqrt{{L_{\edit}} - k + 1}}
        \le
        \frac{\pi}{\sqrt{2{L_{\edit}}}}(c_k+d_k).
    \]
    For $k>{K_{\edit}}$ we use $\theta_k(PaQ, PQ)\le \pi/2$. Thus
    \[
        \dist \rho(PaQ, PQ)
        \le
        \frac{\pi}{\sqrt{2{L_{\edit}}}} \sum_{1\le k\le {K_{\edit}}}\rho^k(c_k+d_k)
        +
        \frac{\pi}{2}\sum_{k={K_{\edit}}+1}^{\infty}\rho^k
        +
        \frac{\pi}{2}\rho^{{L_{\edit}}+1}.
    \]
    For $k\le {K_{\edit}}$, we have $c_k + d_k \le 2\min(k,{M_{\edit}})$, so
    \[
        \sum_{1\le k\le {K_{\edit}}}\rho^k(c_k+d_k)
        \le
        2\sum_{k=1}^{\infty}\rho^k\min(k,{M_{\edit}})
        =
        2A_{M_{\edit}}(\rho).
    \]
    Substituting this and evaluating the tail $\sum_{k={K_{\edit}}+1}^{\infty}\rho^k$ yields \eqref{Equation: insertion and deletion}.

    We now turn our attention to $\dist \rho(P a Q, P b Q)$. Define
    \begin{equation}\label{Definition: ek}
    e_k \coloneqq \max\parent{0,\ \min(m, {L_{\edit}} - k + 1) - \max(0, m - k + 1) + 1}
    \end{equation}
    to be the number of length-$k$ windows in the strings $P a Q$ and $P b Q$ that contain $a$ or $b$ respectively. Each such window changes one $k$-gram into another, so the difference vector is a sum of $e_k$ vectors of the form $u-v$, where $u$ and $v$ are standard basis vectors. Hence, by the triangle inequality,
    \begin{equation}\label{Equation: L2 substitution safe}
        \norm{\ngram {PaQ} k - \ngram{PbQ} k} \leq \sqrt{2}\, e_k.
    \end{equation}
    Since $\ngram{PaQ}k$ and $\ngram{PbQ}k$ have nonnegative integer entries summing to ${L_{\edit}}-k+2$, we have $\min(\norm{\ngram{PaQ}k},\norm{\ngram{PbQ}k})\ge \sqrt{{L_{\edit}}-k+2}$. Thus, by \eqref{Equation: inequality between theta and arcsine},
    \[
        \theta_k(PaQ, PbQ)
        \le
        \pi \cdot \frac{\sqrt{2}\,e_k}{2\sqrt{{L_{\edit}}-k+2}}
        \qquad (k\le {L_{\edit}}+1).
    \]
    Therefore
    \begin{equation}\label{Equation: bound on distance from PaQ to PbQ}
    \dist \rho (P a Q, P b Q)
    \leq 
    \sum_{1 \leq k \leq {L_{\edit}} + 1} \rho^k \cdot \pi \frac{\sqrt{2}\,e_k}{2 \sqrt{{L_{\edit}} - k + 2}}.
    \end{equation}

    For $k\le {K_{\edit}}$ we have ${L_{\edit}}-k+2\ge {L_{\edit}}/2$, so
    \[
        \pi \frac{\sqrt{2}\,e_k}{2 \sqrt{{L_{\edit}} - k + 2}}
        \le
        \frac{\pi}{\sqrt{{L_{\edit}}}}\, e_k.
    \]
    For $k>{K_{\edit}}$ we use $\theta_k(PaQ, PbQ)\le \pi/2$. Hence
    \[
        \dist \rho(PaQ, PbQ)
        \le
        \frac{\pi}{\sqrt{{L_{\edit}}}} \sum_{1\le k\le {K_{\edit}}}\rho^k e_k
        +
        \frac{\pi}{2}\sum_{k={K_{\edit}}+1}^{\infty}\rho^k.
    \]
    Since $e_k \le \min(k,{M_{\edit}})$ for $k\le {K_{\edit}}$, we obtain
    \[
        \sum_{1\le k\le {K_{\edit}}}\rho^k e_k
        \le
        \sum_{k=1}^{\infty}\rho^k\min(k,{M_{\edit}})
        =
        A_{M_{\edit}}(\rho),
    \]
    and evaluating the tail gives \eqref{Equation: substitution}.
\end{proof}

Our next theorem controls the behavior of $\dist \rho$ under stutter.

\begin{theorem}\label{Theorem: stutter}
    Let $\rho \in (0, 1)$. Let $P_1, P_2, Q \in \alphabet^*$ be strings of lengths $m_1$, $m_2$ and $n$ respectively, and let $\ell$ be a positive integer. Then
    \begin{equation}\label{Equation: stutter}
        \dist \rho(P_1QP_2,\, P_1 Q^\ell P_2)
        \leq 
        \frac{\pi}{\sqrt{2L_{\stutter}}}
        \parent{ 3A_{M_{\stutter}}(\rho) + \frac{\rho}{1-\rho}\, r }
        + \frac{\pi}{2}\cdot \frac{\rho^{{K_{\stutter}}+1}}{1-\rho}
        + \frac{\pi}{2} \cdot \frac{\rho^{{L_{\stutter}}+1}(1-\rho^r)}{1-\rho},
    \end{equation}
    where ${L_{\stutter}} \coloneqq m_1 + n + m_2$, $r \coloneqq (\ell - 1) n$, ${K_{\stutter}} \coloneqq \floor{{L_{\stutter}}/2}$, and ${M_{\stutter}} \coloneqq \min(m_1 + n + 1, m_2 + 1)$.
    $A_{M}(\rho)$ is defined as in \eqref{Equation: AM(rho)}.
\end{theorem}

\begin{proof}
    Write $U \coloneqq P_1QP_2$ and $V \coloneqq P_1Q^\ell P_2$, and write $p \coloneqq m_1 + n$ and $q \coloneqq m_2$, so that $\abs U = p+q = {L_{\stutter}}$ and $V$ is obtained from $U$ by inserting the block $B \coloneqq Q^{\ell-1}$ of length $r$ between the prefix of length $p$ and the suffix of length $q$.

    Fix $k\le {L_{\stutter}}$. Define
    \begin{equation}\label{Definition: ck stutter}
        c_k
        \coloneqq
        \max\parent{0,\ \min(p, k - 1) - \max(1, k - q) + 1}
    \end{equation}
    to be the number of length-$k$ windows in $U$ that cross the junction between the prefix of length $p$ and the suffix of length $q$. Define 
    \begin{equation}\label{Definition: dk stutter}
        d_k
        \coloneqq
        \max\parent{0,\ \min(p+r,\, {L_{\stutter}}+r-k+1) - \max(1,\, p-k+2) + 1}
    \end{equation}
    to be the number of length-$k$ windows in $V$ that intersect the inserted block $B$. Only $k$-grams coming from these windows can change when passing from $U$ to $V$. By the same triangle-inequality argument as in \eqref{Equation: L2 insertion safe}, we have
    \begin{equation}\label{Equation: diff bound stutter}
        \norm{\ngram V k - \ngram U k} \leq c_k + d_k.
    \end{equation}
    Since $\ngram U k$ has nonnegative integer entries summing to ${L_{\stutter}}-k+1$, we have $\norm{\ngram U k} \geq \sqrt{{L_{\stutter}}-k+1},$
    and similarly for $\ngram V k$. Therefore, by \eqref{Equation: inequality between theta and arcsine}, for $k \leq {L_{\stutter}}$ we have
    \[
        \theta_k(U,V)
        \le
        \pi\cdot \frac{c_k+d_k}{2\sqrt{{L_{\stutter}}-k+1}}.
    \]

    For $k>{L_{\stutter}}$, we have $\ngram U k =0$, while $\ngram V k\neq 0$ for $k\leq {L_{\stutter}}+r$, so
    \[
        \theta_k(U,V)= \begin{cases} \frac{\pi}{2} & \text{if} \ {L_{\stutter}} < k \leq {L_{\stutter}} + r, \\
        0 & \text{if} k > {L_{\stutter}} + r.
        \end{cases}
    \]
    Hence
    \begin{equation}\label{Equation: sharp stutter}
        \dist \rho(U,V)
        \le
        \sum_{1\le k\le {L_{\stutter}}}\rho^k\cdot \pi\frac{c_k+d_k}{2\sqrt{{L_{\stutter}}-k+1}}
        +
        \frac{\pi}{2}\sum_{k={L_{\stutter}}+1}^{{L_{\stutter}}+r}\rho^k.
    \end{equation}

    Let ${K_{\stutter}}=\floor{{L_{\stutter}}/2}$. For $k\le {K_{\stutter}}$ we have ${L_{\stutter}}-k+1\ge {L_{\stutter}}/2$, so
    \[
        \pi\frac{c_k+d_k}{2\sqrt{{L_{\stutter}}-k+1}}
        \le
        \frac{\pi}{\sqrt{2{L_{\stutter}}}}(c_k+d_k).
    \]
    Moreover,
    \[
        c_k \le \min(k,{M_{\stutter}}),
        \qquad
        d_k \le r + 2\min(k,{M_{\stutter}}),
    \]
    so
    \begin{equation}\label{Equation: c+d stutter}
        c_k + d_k \le r + 3\min(k,{M_{\stutter}})
        \qquad (k\le {K_{\stutter}}).
    \end{equation}
    Therefore
    \[
        \sum_{1\le k\le {K_{\stutter}}}\rho^k(c_k+d_k)
        \le
        r\sum_{k=1}^{{K_{\stutter}}}\rho^k + 3\sum_{k=1}^{{K_{\stutter}}}\rho^k\min(k,{M_{\stutter}})
        \le
        r\frac{\rho}{1-\rho} + 3A_{M_{\stutter}}(\rho).
    \]
    Bounding the remaining $k>{K_{\stutter}}$ part of the first sum in \eqref{Equation: sharp stutter} by $\frac{\pi}{2}\sum_{k={K_{\stutter}}+1}^{\infty}\rho^k$ and evaluating the geometric tails yields \eqref{Equation: stutter}.
\end{proof}

\subsection{Topology}\label{subsec:topology}

In this subsection, we describe the topology induced on $\alphabet^*$ by $\dist \rho$. 

We first note that we can recover $\abs{S}$ from $\dist \rho$.

\begin{lemma}\label{Lemma: We can recover |S| from dist rho}
    Let $\rho > 0$. For any $S \in \alphabet^*$, we have
    \[
    \abs{S} = \begin{cases}
        \frac{\log\parent{1 - \frac{2 (1 - \rho) \dist \rho (S, \epsilon)}{\pi \rho}}}{\log \rho} & \text{if} \ \rho \neq 1, \\
        \frac{2}{\pi} \dist \rho (S, \epsilon) & \text{if} \ \rho = 1.
    \end{cases}
    \]
\end{lemma}

\begin{proof}
    Let $S \in \alphabet^*$, and write $n = \abs S$. We have
    \[
        \dist \rho (S, \epsilon) = \frac{\pi}{2} \sum_{k = 1}^{n} \rho^k, %= \frac{\pi \rho (1 - \rho^n)}{2\parent{1 - \rho}}.
    \]
    which sums as a geometric series if $\rho \neq 1$. Rearranging, we obtain the desired result.
\end{proof}

We can also recover the local angular distances from $\dist \rho$.

\begin{prop}\label{Proposition: theta n is defined in terms of dist rho}
    Let $\rho > 0$, and let $n \geq 1$. For any $S, T \in \alphabet^*$, there is an explicit recursive formula expressing $\theta_n(S,T)$ in terms of the finite set of distances
    \[
    \{\dist \rho(S,W),\ \dist \rho(T,W): W\in\alphabet^{\leq n}.\}
    \]
    Consequently, the collection of functions $(\theta_n)_{n\ge 1}$ is determined by $\dist \rho$.
\end{prop}

\begin{proof}
    Let $\rho > 0$. We proceed by induction on $n$. Note that for any $S$, and for $W$ of length $n$, we have 
    \[
    \ngram S n = \sum_{W \in \alphabet^n} \multiplicity S W \ngram W n = \sum_{W \in \alphabet^n} \norm{\ngram S n} \cos \theta_n (S, W) \ngram W n.
    \]
    Thus for any $S, T \in \alphabet^*$ we have
    \[
    \theta_n (S, T) = \arccos\parent{\sum_{W \in \alphabet^n} \cos \theta_n(S, W) \cos \theta_n (T, W)}
    \]
    
    It therefore suffices to evaluate $\theta_n(S, W)$ for $W$ of length $n$. Suppose now that we have performed such an evaluation for all $k < n$. The evaluation set is empty when $n=1$, giving the base case. We may write
    \begin{align*}
    \theta_n(S, W) &= \frac{\dist \rho (S, W) - \sum_{k < n} \rho^k \theta_k(S, W) - \frac{\pi}{2} \sum_{n < k \leq \abs{S}} \rho^k}{\rho^n} \\
    &= \frac{\dist \rho (S, W) - \sum_{k < n} \rho^k \theta_k(S, W) - \max\parent{0, \dist \rho (\epsilon, S) - \dist \rho (\epsilon, W)}}{\rho^n},
    \end{align*}
    and our claim follows.
\end{proof}

\begin{remark}
The proof gives an explicit algorithm: knowing $\dist \rho$, one first recovers lengths via $\dist \rho(\epsilon,\cdot)$, then computes $\theta_1$, then $\theta_2$, and so on.
\end{remark}

\begin{example}
    For $n = 1$, $S, T \in \alphabet^*$ nonempty, and $a \in \alphabet$, \Cref{Proposition: theta n is defined in terms of dist rho} gives us the following decompositions of $\theta_1(S, a)$ and $\theta_1(S, T)$:
    \begin{align*}
        \theta_1(S, a) &= \frac{\dist \rho (S, a) - \parent{\dist \rho (\epsilon, S) - \frac{\pi}{2} \rho}}{\rho}, \\
        \theta_1(S, T) &= \arccos\parent{\sum_{a \in \alphabet} \frac{\parent{\dist \rho (S, a) - \parent{\dist \rho (\epsilon, S) - \frac{\pi}{2} \rho}}\parent{\dist \rho (T, a) - \parent{\dist \rho (\epsilon, T) - \frac{\pi}{2} \rho}}}{\rho^2}}.
    \end{align*}
\end{example}

It turns out that $\dist \rho$ induces the discrete topology on $\alphabet^*$. For the following proposition, we relax the condition $\rho \in (0, 1)$.

\begin{proposition}\label{Proposition: discrete topology on strings}
    Let $\rho>0$. For every $S\in\alphabet^*$ there exists $\delta(S,\rho)>0$ such that
    \[
        T\neq S \ \text{implies} \ \dist\rho(S,T) \ge \delta(S,\rho).
    \]
    In particular, the metric space $(\alphabet^*,\dist\rho)$ has the discrete topology: for
    each $S$ the open ball $B_{\delta(S,\rho)/2}(S)$ equals $\{S\}$.
\end{proposition}

\begin{proof}
    We consider three cases according to the length of $T$.

    If $\abs{T}<\abs{S}$, then $\ngram {T} {\abs{S}}=0$ while $\ngram {S} {\abs{S}}$ has a single nonzero coordinate, so $\theta_{\abs{S}}(S,T)=\frac{\pi}{2}$ and
    \[
        \dist\rho(S,T)\ \ge\ \rho^{\abs{S}} \frac{\pi}{2}.
    \]

    If $\abs{T}=\abs{S}$ and $T\neq S$, then $\ngram {S} {\abs{S}}$ and $\ngram {T} {\abs{S}}$ are distinct standard basis vectors in $\bbR^{\alphabet^{\abs{S}}}$, hence orthogonal.  Thus again $\theta_{\abs{S}}(S,T)=\frac{\pi}{2}$ and
    \[
        \dist\rho(S,T)\ \ge\ \rho^{\abs{S}} \frac{\pi}{2}.
    \]

    If $\abs{T}>\abs{S}$, then $\ngram S {\abs{S}+1}=0$ while $\ngram T {\abs{S}+1}\neq 0$, so
    $\theta_{\abs{S}+1}(S,T)=\frac{\pi}{2}$ and
    \[
        \dist\rho(S,T)\ \ge\ \rho^{\abs{S}+1} \frac{\pi}{2}.
    \]

    In all cases
    \[
        \dist\rho(S,T)\ \ge\ \frac{\pi}{2}\,\min\bigl(\rho^{\abs{S}},\rho^{\abs{S}+1}\bigr)
        \eqqcolon \delta(S,\rho) > 0,
    \]
    and the claim follows.
\end{proof}

\begin{remark}
    When $\rho \geq 1$, the proof of \Cref{Proposition: discrete topology on strings} implies that $\dist \rho$ is uniformly discrete, and in fact that $B_{\pi/2}(S)$ equals $\{S\}$ for any $S \in \alphabet^*$.
\end{remark}

\Cref{Proposition: discrete topology on strings} is not the full story, however: the topology induced by $\rho$ is only \emph{uniformly} discrete when $\rho \geq 1$. When $0 < \rho < 1$, $\alphabet^*$ has a nontrivial completion.

\begin{definition}\label{Definition: shift invariant measures}
    Let $\alphabet$ be a finite alphabet and let $\alphabet^{\bbN}$ be the one-sided full shift,
    equipped with the product topology, and let $\shift : x_1 x_2 \dots \mapsto x_2 x_3 \dots$ be the left shift.
    (See \cite[Chapter~1]{LindMarcus2021} for the standard shift-space topology, and \cite[\S 1.1]{Shields1996}
    for the one-sided Kolmogorov/measure-theoretic viewpoint; in particular, cylinder sets generate the Borel
    $\sigma$-algebra.)

    A \defi{shift-invariant Borel probability measure} on $\alphabet^{\bbN}$ is a Borel probability
    measure $\mu$ such that $\mu\circ \shift^{-1}=\mu$. We write $\mathcal M_{\shift}(\alphabet^{\bbN})$
    for the set of such measures.

    For $n\ge 1$ and $W\in\alphabet^n$, define the \defi{cylinder set}
    \[
        \mathrm{Cyl}(W)\coloneqq \{x\in\alphabet^{\bbN} : x_1x_2\cdots x_n = W\},
    \]
    and the length-$n$ \defi{cylinder marginal}
    \[
        p_n^\mu(W)\coloneqq \mu(\mathrm{Cyl}(W)).
    \]
\end{definition}

Shift-invariant measures are a natural receptacle for limiting $n$-gram statistics.
A finite string can be viewed as a short sample from a symbolic process, while a
shift-invariant measure records the translation-stable distribution of all finite blocks.
In particular, for $\mu\in\mathcal M_{\shift}(\alphabet^{\bbN})$ the marginals satisfy the
consistency relations
\[
    \sum_{a\in\alphabet} p_{n+1}^\mu(Wa)=p_n^\mu(W)
    \qquad\text{and}\qquad
    \sum_{a\in\alphabet} p_{n+1}^\mu(aW)=p_n^\mu(W),
\]
where the first identity is just disjoint union of cylinders, and the second follows from
$\mu(\mathrm{Cyl}(W))=\mu(\shift^{-1}\mathrm{Cyl}(W))$ together with
$\shift^{-1}\mathrm{Cyl}(W)=\bigsqcup_{a\in\alphabet}\mathrm{Cyl}(aW)$
(see \cite[\S 1.1]{Shields1996}).

We now describe a common extension of $\dist\rho$ to both finite strings and shift-invariant
measures.

\begin{definition}\label{Definition: extended thetas and distance}
    Fix $\rho\in(0,1)$.  For any finite string $S$ of length $L$ and any $n\ge 1$, define its
    empirical length-$n$ block distribution by
    \[
        p_n^S(W)\coloneqq
        \begin{cases}
            \dfrac{\multiplicity S W}{L-n+1}, & L\ge n,\\[6pt]
            0, & L<n,
        \end{cases}
        \qquad (W\in\alphabet^n),
    \]
    and let $u_n(S)$ be the $L^2$-normalization of the vector $p_n^S\in \bbR^{\alphabet^n}_{\ge 0}$
    when it is nonzero, with the convention $u_n(S)=0$ otherwise.

    For $\mu\in\mathcal M_{\shift}(\alphabet^{\bbN})$, define
    \[
        p_n^\mu(W)\coloneqq \mu(\mathrm{Cyl}(W)),
        \qquad
        u_n(\mu)\coloneqq \frac{p_n^\mu}{\norm{p_n^\mu}_2}.
    \]
    (Note that $p_n^\mu$ is a probability vector, hence nonzero for every $n\ge 1$.)

    For $X,Y\in \alphabet^* \sqcup \mathcal M_{\shift}(\alphabet^{\bbN})$ and $n\ge 1$, define
    \[
        \theta_n(X,Y)\coloneqq
        \begin{cases}
            \arccos\!\parent{\langle u_n(X),u_n(Y)\rangle}, & u_n(X),u_n(Y)\neq 0,\\
            0, & u_n(X)=u_n(Y)=0,\\
            \dfrac{\pi}{2}, & \text{exactly one of }u_n(X),u_n(Y)\text{ is }0,
        \end{cases}
    \]
    and set
    \[
        \widehat{\dist \rho}(X,Y)\coloneqq \sum_{n=1}^{\infty}\rho^n \theta_n(X,Y).
    \]
\end{definition}

The next lemma explains why this really extends the finite-string construction.

\begin{lemma}\label{Lemma: extension agrees with original}
    For a finite string $S$ of length $m\ge n$ we have
    \[
        p_n^S = \frac{\ngram S n}{m-n+1}.
    \]
    In particular, $u_n(S)$ agrees with the $L^2$-normalization of $\ngram S n$, and for finite
    strings $S,T$ the angle $\theta_n(S,T)$ from \Cref{Definition: extended thetas and distance}
    coincides with the $n$-gram angle used in \Cref{Definition: weighted angle distance}.
    Consequently, for $S,T\in\alphabet^*$,
    \[
        \widehat{\dist \rho}(S,T) = \dist\rho(S,T).
    \]
\end{lemma}

\begin{proof}
    If $m\ge n$, then the number of length-$n$ windows in $S$ is $m-n+1$, and each such window
    contributes one to the appropriate coordinate in $\ngram S n$.  Thus
    \[
        \sum_{W\in\alphabet^n} p_n^S(W)
        = \frac{1}{m-n+1}\sum_W \multiplicity SW = 1,
    \]
    so $p_n^S$ is exactly the normalized count vector.  The $L^2$-normalization is therefore
    the same in both constructions, and the definitions of $\theta_n$ coincide on finite strings.
    Summing with weights $\rho^n$ yields the final claim.
\end{proof}

\begin{lemma}\label{Lemma: extended distance is a metric}
    The function $\widehat{\dist \rho}$ is a metric on
    $\alphabet^* \sqcup \mathcal M_{\shift}(\alphabet^{\bbN})$.
\end{lemma}

\begin{proof}
    Fix $n$.  On the set $\{0\}\cup S(\bbR^{\alphabet^n})$ (unit sphere plus an extra point $0$),
    the rule defining $\theta_n$ is a metric: it is the spherical metric on the unit sphere, with
    the extra point $0$ placed at distance $\pi/2$ from every unit vector and distance $0$ from itself.
    Therefore $\theta_n$ is a metric on the disjoint union
    $\alphabet^* \sqcup \mathcal M_{\shift}(\alphabet^{\bbN})$ when composed with $u_n(\cdot)$.
    The weighted sum $\widehat{\dist\rho}=\sum_{n\ge 1}\rho^n\theta_n$ is therefore a metric
    provided it separates points.

    Separation for finite--finite follows from \Cref{Lemma: extension agrees with original} and
    \Cref{Proposition: dist rho is a metric again}.

    Separation for measure--measure: if $\widehat{\dist\rho}(\mu,\nu)=0$, then $\theta_n(\mu,\nu)=0$
    for all $n$, hence $u_n(\mu)=u_n(\nu)$ for all $n$. Since $p_n^\mu$ and $p_n^\nu$ are probability vectors,
    $u_n(\mu)=u_n(\nu)$ implies $p_n^\mu=p_n^\nu$ for all $n$ (because $p=u/\|u\|_1$ on the nonnegative orthant).
    Cylinder sets generate the Borel $\sigma$-algebra, so $\mu=\nu$ (see \cite[\S 1.1]{Shields1996}).

    Separation for mixed pairs: if $S$ has length $L$ and $\mu\in\mathcal M_{\shift}(\alphabet^{\bbN})$, then
    $u_n(S)=0$ for all $n>L$ while $u_n(\mu)\neq 0$ for all $n$, so $\theta_n(S,\mu)=\pi/2$ for every $n>L$ and
    \[
        \widehat{\dist \rho}(S,\mu)\ge \frac{\pi}{2}\sum_{n>L}\rho^n>0.
    \]
\end{proof}

\begin{lemma}\label{Lemma: cauchy coordinates}
    Let $(X_j)_{j\ge 1}$ be a $\widehat{\dist \rho}$-Cauchy sequence in
    $\alphabet^* \sqcup \mathcal M_{\shift}(\alphabet^{\bbN})$.
    Then for each fixed $n$ the sequence $(X_j)$ is $\theta_n$-Cauchy, and the vectors
    $u_n(X_j)$ form a Cauchy sequence in $\{0\}\cup S(\bbR^{\alphabet^n})$.
\end{lemma}

\begin{proof}
    For any $i,j$ and any $n$ we have
    \[
        \rho^n \theta_n(X_i,X_j) \le \widehat{\dist\rho}(X_i,X_j).
    \]
    Since the right-hand side tends to $0$ as $i,j\to\infty$, we have
    $\theta_n(X_i,X_j)\to 0$ for each $n$, i.e.\ $(u_n(X_j))$ is Cauchy in the metric space
    $\{0\}\cup S(\bbR^{\alphabet^n})$.
\end{proof}

\begin{lemma}\label{Lemma: existence of measure limits}
    Let $(X_j)$ be $\widehat{\dist\rho}$-Cauchy.
    \begin{enumerate}
        \item If infinitely many $X_j$ are finite strings and their lengths are bounded, then
              $(X_j)$ is eventually constant, and hence converges to a finite string.\label{Item: eventually constant Cauchy sequences}
        \item Otherwise, there exists $\mu\in\mathcal M_{\shift}(\alphabet^{\bbN})$ such that
              $\widehat{\dist\rho}(X_j,\mu)\to 0$.\label{Item: other Cauchy sequences}
    \end{enumerate}
\end{lemma}

\begin{proof}
    \eqref{Item: eventually constant Cauchy sequences}
    If $\abs{X_j}$ takes only finitely many values and infinitely many $X_j$ are strings, then
    there are only finitely many possible strings among those terms.  A Cauchy sequence in a finite
    metric space must be eventually constant.

    \eqref{Item: other Cauchy sequences}
    We may therefore assume that either $X_j$ are measures for all large $j$, or that $X_j$ are
    finite strings with $\abs{X_j}\to\infty$.

    Fix $n\ge 1$.  By \Cref{Lemma: cauchy coordinates} the sequence $u_n(X_j)$ converges in
    $\{0\}\cup S(\bbR^{\alphabet^n})$.  In our present case, for each fixed $n$ we eventually have
    $\abs{X_j}\ge n$ (if $X_j$ are strings) or $X_j\in\mathcal M_{\shift}$ (if they are measures), hence
    $u_n(X_j)\neq 0$ for all large $j$.  Therefore the limit lies in the unit sphere; write it as $u_n^\infty$.

    Define probability vectors on $\alphabet^n$ by
    \[
        p_n(W)\coloneqq \frac{u_n^\infty(W)}{\sum_{U\in\alphabet^n}u_n^\infty(U)},\qquad W\in\alphabet^n.
    \]
    Since $u_n^\infty\ge 0$ and $\|u_n^\infty\|_2=1$, the denominator is positive.
    Moreover, the map $u\mapsto u/\|u\|_1$ is continuous on the nonnegative unit sphere, so
    \[
        p_n^{X_j}=\frac{u_n(X_j)}{\|u_n(X_j)\|_1}\longrightarrow p_n
    \]
    for each fixed $n$.

    We claim that $(p_n)_{n\ge 1}$ satisfies the Kolmogorov consistency relations, i.e.
    $\sum_{a}p_{n+1}(Wa)=p_n(W)$ and $\sum_{a}p_{n+1}(aW)=p_n(W)$.
    For $X_j\in\mathcal M_{\shift}$ these equalities hold exactly (as discussed after
    \Cref{Definition: shift invariant measures}), so they pass to the limit immediately.
    For $X_j$ a finite string $S$ of length $L\ge n+1$, the corresponding empirical distributions satisfy
    the approximate identities
    \[
      \Bigl|\sum_{a\in\alphabet}p_{n+1}^{S}(Wa)-p_n^{S}(W)\Bigr|
      \le \frac{1}{L-n+1},
      \qquad
      \Bigl|\sum_{a\in\alphabet}p_{n+1}^{S}(aW)-p_n^{S}(W)\Bigr|
      \le \frac{1}{L-n+1},
    \]
    because the $(n+1)$-windows whose prefix is $W$ account for all $n$-windows equal to $W$
    except possibly the last one, and similarly for suffixes (a boundary-term estimate).
    Since $\abs{X_j}\to\infty$ in the string case, the error tends to $0$ with $j$, and passing to the
    limit gives the exact consistency relations for $(p_n)$.

    By the Kolmogorov extension theorem (see \cite[Theorem~1.1.2]{Shields1996}), there is a unique
    Borel probability measure $\mu$ on $\alphabet^{\bbN}$ such that $\mu(\mathrm{Cyl}(W))=p_n(W)$ for
    all $n$ and all $W\in\alphabet^n$.  The second consistency relation is exactly the stationarity
    (shift-invariance) condition for the one-sided Kolmogorov model (see \cite[\S 1.1]{Shields1996}),
    hence $\mu\in\mathcal M_{\shift}(\alphabet^{\bbN})$.

    Finally, by construction $u_n(\mu)=u_n^\infty$ for each $n$, so $\theta_n(X_j,\mu)\to 0$ for each fixed $n$.
    Since $0\le \theta_n\le \pi/2$, given $\varepsilon>0$ choose $N$ with
    $\frac{\pi}{2}\sum_{n>N}\rho^n<\varepsilon/2$, and then choose $j$ large enough that
    $\sum_{n\le N}\rho^n\theta_n(X_j,\mu)<\varepsilon/2$.  This yields $\widehat{\dist\rho}(X_j,\mu)<\varepsilon$,
    proving $\widehat{\dist\rho}(X_j,\mu)\to 0$.
\end{proof}

\begin{lemma}\label{Lemma: density of finite strings}
    For every $\mu\in\mathcal M_{\shift}(\alphabet^{\bbN})$ and every $\varepsilon>0$, there exists
    a finite string $S\in\alphabet^*$ such that $\widehat{\dist\rho}(S,\mu)<\varepsilon$.
\end{lemma}

\begin{proof}
    Fix $\mu$ and $\varepsilon>0$.  Choose $N$ such that
    \[
        \frac{\pi}{2}\sum_{n>N}\rho^n < \frac{\varepsilon}{2}.
    \]
    It suffices to construct $S$ so that $\theta_n(S,\mu)$ is small for each $1\le n\le N$.

    Consider the de Bruijn digraph $G_N$ with vertex set $\alphabet^{N-1}$ and edge set $\alphabet^N$,
    where an edge $W=w_1\cdots w_N$ goes from the prefix $w_1\cdots w_{N-1}$ to the suffix $w_2\cdots w_N$.
    The vector $p_N^\mu$ is a nonnegative circulation on $G_N$: for each vertex $V\in\alphabet^{N-1}$,
    shift-invariance implies
    \[
      \sum_{a\in\alphabet}p_N^\mu(Va)=p_{N-1}^\mu(V)=\sum_{a\in\alphabet}p_N^\mu(aV).
    \]

    Decompose this circulation into directed cycles with nonnegative coefficients (a standard flow
    decomposition: iteratively follow edges of positive weight until a directed cycle appears, subtract
    the minimum weight on that cycle, and repeat).  Approximating the finitely many cycle coefficients by
    rationals yields a rational circulation $q_N$ on $\alphabet^N$ that is arbitrarily close to $p_N^\mu$
    in $\ell^2$, and still satisfies the same balance equations (hence both consistency relations).

    Choose a common denominator $M$ so that $Q(W)\coloneqq M q_N(W)\in\bbN$ for all $W\in\alphabet^N$.
    Build the directed multigraph with $Q(W)$ parallel edges labeled by $W$ as above.  Balance of $q_N$
    implies that every vertex has indegree equal to outdegree, so the multigraph has an Eulerian cycle.
    Reading symbols along such a cycle produces a finite string $S_0$ of length $M+N-1$ whose length-$N$
    empirical block distribution is exactly $q_N$, i.e.\ $p_N^{S_0}=q_N$.

    Now take $S\coloneqq S_0^k$ (concatenate $k$ copies) with $k$ large.  For each $1\le n\le N$,
    the boundary-term estimates used in the proof of \Cref{Lemma: existence of measure limits} show that
    $p_n^{S}$ is within $O(1/k)$ (in $\ell^1$, hence in $\ell^2$) of the corresponding stationary marginal of $q_N$,
    and therefore (by choosing $q_N$ close enough to $p_N^\mu$ and then $k$ large enough) we can ensure that
    $p_n^{S}$ is arbitrarily close to $p_n^\mu$ for every $1\le n\le N$.

    Since the map $p\mapsto p/\|p\|_2$ is continuous on the probability simplex (and $\|p\|_2\ge |\alphabet|^{-n/2}$),
    closeness of $p_n^{S}$ to $p_n^\mu$ implies $\theta_n(S,\mu)$ is as small as we like for $1\le n\le N$.
    Hence $\sum_{n\le N}\rho^n\theta_n(S,\mu)<\varepsilon/2$, and the tail $n>N$ contributes at most $\varepsilon/2$
    by the choice of $N$.  Therefore $\widehat{\dist\rho}(S,\mu)<\varepsilon$.
\end{proof}

\begin{theorem}[Completion of $\alphabet^*$]\label{Theorem: Completion of alphabet^*}
    Let $\rho\in(0,1)$, and let $\widehat{\dist\rho}$ be as in
    \Cref{Definition: extended thetas and distance}.  Then:
    \begin{enumerate}
        \item $\widehat{\dist \rho}$ is a metric on
              $\alphabet^* \sqcup \mathcal M_{\shift}(\alphabet^{\bbN})$.\label{Item: widehat dist rho is a metric}
        \item The restriction of $\widehat{\dist \rho}$ to $\alphabet^*$ equals $\dist \rho$.\label{Item: widehat dist rho restricts to dist rho}
        \item The metric space
              $\bigl(\alphabet^* \sqcup \mathcal M_{\shift}(\alphabet^{\bbN}),\widehat{\dist \rho}\bigr)$
              is complete, and $\alphabet^*$ is dense in it.\label{Item: we have a complete space}
    \end{enumerate}
    Consequently, the completion of $\bigl(\alphabet^*,\dist \rho\bigr)$ is naturally identified with
    the disjoint union
    \[
        \alphabet^* \sqcup \mathcal M_{\shift}(\alphabet^{\bbN})
    \]
    equipped with $\widehat{\dist \rho}$.
\end{theorem}

\begin{proof}
    \eqref{Item: widehat dist rho is a metric} is \Cref{Lemma: extended distance is a metric}.
    \eqref{Item: widehat dist rho restricts to dist rho} is \Cref{Lemma: extension agrees with original}.
    Completeness and existence of limits follow from \Cref{Lemma: cauchy coordinates,Lemma: existence of measure limits},
    and density is \Cref{Lemma: density of finite strings}; together these give \eqref{Item: we have a complete space}.
    The final identification with the metric completion follows from the universal property of completions.
\end{proof}

Note that the completion of $\alphabet^*$ is independent of our choice of $\rho$. This is not a coincidence.

\begin{theorem}\label{Theorem: topological equivalence and inequivalence of different rho}
   Let $1 > \sigma > \rho > 0$. The identity map is a uniform homeomorphism from
   $(\alphabet^*, \dist \sigma)$ to $(\alphabet^*, \dist \rho)$.
\end{theorem}

\begin{proof}
    For any $S, T \in \alphabet^*$ we have
    \[
    \dist \sigma (S, T) = \sum_{n \geq 1} \sigma^n \theta_n(S, T)
    \geq \sum_{n \geq 1} \rho^n \theta_n(S, T) = \dist \rho (S, T),
    \]
    so the identity map from $(\alphabet^*, \dist \sigma)$ to $(\alphabet^*, \dist \rho)$ is $1$-Lipschitz.

    For the reverse direction, fix $\epsilon > 0$.  For any integer $N\ge 1$,
    \begin{align*}
    \dist \sigma(S, T)
    &= \sum_{1 \leq n < N} \sigma^n \theta_n(S, T) + \sum_{n \geq N} \sigma^n \theta_n(S, T) \\
    &\leq \sum_{1 \leq n < N} \frac{\sigma^n}{\rho^n} \dist \rho (S, T) + \frac{\pi}{2} \cdot \frac{\sigma^N}{1 - \sigma},
    \end{align*}
    since $0\le \theta_n\le \pi/2$ on finite strings (all $u_n$ are nonnegative).
    Choose $N$ so that $\frac{\pi}{2}\frac{\sigma^N}{1-\sigma}<\epsilon/2$, and then choose
    \[
      \delta < \frac{\epsilon}{2 \sum_{1 \leq n < N} (\sigma^n/\rho^n)}.
    \]
    Then $\dist \rho(S,T)<\delta$ implies $\dist \sigma(S,T)<\epsilon$, proving uniform continuity.
    The inverse is continuous as well, hence the identity is a uniform homeomorphism.
\end{proof}

\begin{corollary}\label{Cor: uniform equivalence extends to the completion}
   Let $1>\sigma>\rho>0$. The identity on the common underlying set
   $\alphabet^* \sqcup \mathcal M_{\shift}(\alphabet^{\bbN})$ is a uniform homeomorphism between
   $\bigl(\,\cdot\,,\widehat{\dist \sigma}\bigr)$ and $\bigl(\,\cdot\,,\widehat{\dist \rho}\bigr)$.
   In particular, the induced identification between the two metric completions agrees with
   \Cref{Theorem: Completion of alphabet^*}.
\end{corollary}

\begin{proof}
   The same estimates as in \Cref{Theorem: topological equivalence and inequivalence of different rho}
   apply verbatim because $0\le \theta_n\le \pi/2$ on the entire disjoint union by definition.
\end{proof}

\subsection{Isometries}

Fix $\rho \in (0, 1)$. In this subsection, we characterize the isometries of $\dist \rho$.

Let $\pi : \alphabet \to \alphabet$ be a permutation, and extend the action of $\pi$ to $\alphabet^*$ character-wise, so $\pi(s_1 \dots s_n) = \pi(s_1) \dots \pi(s_n)$. The map $\pi : \alphabet^* \to \alphabet^*$ is visibly an isometry. Let $R : \alphabet^* \to \alphabet^*$ be the \defi{reversal map} given by $R(s_1 \dots s_n) = s_n \dots s_1$. The reversal map is an isometry of $\alphabet^*$ that commutes with $\pi$. It is the goal of this section to prove that every isometry of $\alphabet^*$ is either a permutation or the product of a permutation and a reversal. Along the way, we prove three lemmas.

\begin{lemma}\label{Lemma: isometries fix epsilon}
    Let $\rho \in (0,1)$. Let $\phi:\alphabet^*\to\alphabet^*$ be an isometry.
    Then $\phi(\epsilon)=\epsilon$.
\end{lemma}

\begin{proof}
    For any $X\in\alphabet^*$, consider the set of all distances from $X$:
    \[
        D_X \coloneqq \{ \dist \rho(X,T) : T\in\alphabet^* \}.
    \]
    Since $\phi$ is an isometry, we have $D_{\phi(X)}=D_X$.

    We readily compute
    \[
        D_\epsilon
        =
        \left\{
        \frac{\pi}{2}\sum_{k=1}^{N}\rho^k
        : N\in\bbZ_{\geq 0}
        \right\},
    \]
    which is a strictly increasing sequence with unique accumulation point $\frac{\pi}{2}\cdot \frac{\rho}{1-\rho}.$

    Now let $S\neq \epsilon$.
    Consider the sequence of strings $S^N$ for $N\ge 1$.
    Observe that $\ngram {S^N} 1 = N \ngram S 1,$ so $\theta_1(S,S^N)=0$.
    For every $k\ge 2$, we trivially have $\theta_k(S,S^N)\le \pi/2$.
    Hence
    \[
        \dist \rho(S,S^N)
        =
        \sum_{k\ge 1}\rho^k \theta_k(S,S^N)
        \le
        \frac{\pi}{2}\sum_{k\ge 2}\rho^k = \frac{\pi}{2} \cdot \frac{\rho^2}{1 - \rho}.
    \]
    The bounded sequence $\parent{\dist \rho(S,S^N)}_{N\ge 1}$ therefore has an accumulation point no greater than $\frac{\pi}{2}\cdot \frac{\rho}{1-\rho}- \frac{\pi}{2} \cdot \rho$.
    In particular, $D_S$ has an accumulation point strictly smaller than the only accumulation point of $D_\epsilon$. Thus $D_S$ cannot equal $D_\epsilon$ for any $S\neq\epsilon$, and we conclude $\phi(\epsilon)=\epsilon$.
\end{proof}

\begin{lemma}\label{Lemma: isometries fix length}
    Let $\rho \in (0, 1)$. Let $\phi : \alphabet^* \to \alphabet^*$ be an isometry. Then for any $S \in \alphabet^*$, we have $\abs{\phi(S)} = \abs{S}$.
\end{lemma}

\begin{proof}
    The claim is immediate from \Cref{Lemma: We can recover |S| from dist rho} and \Cref{Lemma: isometries fix epsilon}.
\end{proof}

\begin{lemma}\label{Lemma: isometries fix theta n}
    Let $\rho \in (0, 1)$. Let $S, T \in \alphabet^*$, let $n \in \bbN$, and let $\phi : \alphabet^* \to \alphabet^*$ be an isometry. We have $\theta_n(S, T) = \theta_n(\phi(S), \phi(T))$.
\end{lemma}

\begin{proof}
    This is immediate from \Cref{Proposition: theta n is defined in terms of dist rho}, \Cref{Lemma: isometries fix epsilon}, and \Cref{Lemma: isometries fix length}.
\end{proof}

\begin{theorem}\label{Theorem: isometries of dist rho}
    Let $\rho \in (0, 1)$. Every isometry of $\alphabet^*$ is either a character-permutation, or the composition of a character-permutation with reversal.
\end{theorem}

\begin{proof}
    Let $\rho \in (0, 1)$. If $\abs{\alphabet} = 1$, then the claim is trivial and immediate from \Cref{Lemma: isometries fix length}. So suppose $\abs{\alphabet} \geq 2$.

    Let $\phi$ be an isometry of $\alphabet^*$ under $\dist \rho$. By \Cref{Lemma: isometries fix length}, $\phi$ induces a permutation $\pi$ on $\alphabet$; extending $\pi$ to $\alphabet^*$ character-wise induces another isometry on $\alphabet^*$, and $\pi\inv \phi$ is an isometry which fixes $\alphabet$. Without loss of generality, we now assume that $\phi$ fixes $\alphabet$. We wish to show that $\phi$ is either the identity map or the reversal map $R$.

    We claim at least that $\phi(a^k) = a^k$ for any character $a \in \alphabet$ and any $k \geq 1$. Indeed, by \Cref{Proposition: theta n is defined in terms of dist rho}, we see 
    \[
    \theta_1(\phi(a^k), a) = \theta_1(\phi(a^k), \phi(a)) = \theta_1(a^k, a) = 0,
    \]
    but $a^k$ is the only length $k$ string with the property $\theta_1(a^k, a) = 0$, so \Cref{Lemma: isometries fix length} implies $\phi(a^k) = a^k$.

    Let $a, b \in \alphabet$ be distinct characters. By \Cref{Lemma: isometries fix theta n}, $\phi$ preserves $\theta_1$, so $\theta_1(\phi(ab), a) = \theta_1(ab, a)$ and $\theta_1(\phi(ab), b) = \theta_1(ab, b)$. As $\abs{\phi(ab)} = 2$, we conclude that $\ngram {\phi(ab)} 1 = \ngram {ab} 1$, so $\phi(ab) \in \set{ab, ba}$. For $a \neq b$, define 
    \[
    \chi(a, b) \coloneqq \begin{cases} 1 & \text{if} \ \phi(ab) = ab, \\
    -1 & \text{if} \ \phi(ab) = ba.
    \end{cases}
    \]
    Note that by bijectivity, $\chi(a, b) = \chi(b, a)$. We claim that $\chi(a, b)$ is independent of our choice of $a \neq b$. If $\abs{\alphabet} = 2$, this claim is trivial. Otherwise, let $a, b, c \in \alphabet$ be distinct characters, and consider the string $abc$, with bigrams $ab$ and $bc$. We have $\theta_2(abc, ab), \theta_2(abc, bc) < \frac{\pi}{2}$ while $\theta_2(abc, xy) = \frac{\pi}{2}$ for all other bigrams $xy$, so $\set{\phi(ab), \phi(bc)}$ must be the bigrams of $\phi(abc)$. As a length 3 string has exactly two bigrams, we see $\set{\phi(ab), \phi(bc)}$ must be the bigrams of some three-letter string. But this implies $\chi(a, b) = \chi(b, c)$ as desired, and there must exist a $\chi \in \set{\pm 1}$ such that $\phi(ab) = \begin{cases} ab & \text{if} \ \chi = 1, \\ ba & \text{if} \ \chi = -1 \end{cases}$ independent of $a$ and $b$.
    
    Suppose by induction that we are given $N$ and $\chi \in \set{\pm 1}$ such that $\phi$ agrees with $R$ on $\alphabet^{\leq N}$ if $\chi = -1$, and $\phi$ agrees with $id$ on $\alphabet^{\leq N}$ if $\chi = 1$. Let $S = s_1 \dots s_{N+1}$ be a string of length $N + 1$ in $\alphabet^*$, and let $T_1 = s_1 \dots s_{N}$ and $T_2 = s_2 \dots s_{N+1}$ be the two $N$-grams of $S$. By $\theta_N$-invariance, we see $\theta_N(\phi(S), \phi(T_1)), \theta_N(\phi(S), \phi(T_2)) < \frac{\pi}{2}$, so $\phi(T_1)$ and $\phi(T_2)$ are the two $N$-grams of $\phi(S)$. Now if $\chi = -1$, then $\phi(T_1) = R(T_1)$ and $\phi(T_2) = R(T_2)$, which uniquely determines $\phi(S) = R(S)$. Likewise, if $\chi = 1$, then $\phi(T_1) = T_1$ and $\phi(T_2) = T_2$, which uniquely determines $\phi(S) = S$. The claim follows.
\end{proof}

\section{Appendix: Tables and plots}

\subsection{Tables and plots}

For completeness we include the main numeric summaries and plots generated from \texttt{results.csv} by our analysis script.

\begin{table}[t]
\centering
\begin{tabular}{l r c c c c}
\toprule
Dataset & $N$ & \textbf{Edit} & \textbf{$k$-gram angle} & \textbf{JS $k$-gram} & \textbf{Weighted angle} \\
\midrule
\texttt{splice} & 3190 & 0.0003/0.0009 & 0.0003/0.0009 & 0.0003/0.0009 & 0.0003/0.0009 \\
\texttt{strseq} & 1624 & 0.8825/0.8970 & 0.3293/0.6725 & 0.5394/0.7714 & 0.5060/0.7527 \\
\texttt{ucsc\_trf} & 8905 & 0.0003/0.0210 & 0.4390/0.8421 & 0.4309/0.8378 & 0.4596/0.8632 \\
\end{tabular}
\caption{Best \emph{ARI/NMI} achieved by each distance family on each dataset (max over the family parameter: $\rho$ for weighted angle, $k$ for $k$-gram methods).}
\label{tab:best_by_family}
\end{table}

\begin{table}[t]
\centering
\begin{tabular}{l l r r r}
\toprule
Dataset & Metric & Best $\rho$ & Best value & Std./Range \\
\midrule
\texttt{splice} & \texttt{ari} & 0.6 & 0.0003 & 0.0003/0.0008 \\
\texttt{splice} & \texttt{nmi} & 0.6 & 0.0009 & 0.0002/0.0004 \\
\texttt{splice} & \texttt{silhouette} & 0.1 & -1.0000 & 0.0000/0.0000 \\
\texttt{strseq} & \texttt{ari} & 0.6 & 0.5060 & 0.1348/0.4105 \\
\texttt{strseq} & \texttt{nmi} & 0.6 & 0.7527 & 0.1066/0.3083 \\
\texttt{strseq} & \texttt{silhouette} & 0.1 & 0.8387 & 0.0722/0.2579 \\
\texttt{ucsc\_trf} & \texttt{ari} & 0.6 & 0.4596 & 0.0817/0.2543 \\
\texttt{ucsc\_trf} & \texttt{nmi} & 1.0 & 0.8643 & 0.0442/0.1463 \\
\texttt{ucsc\_trf} & \texttt{silhouette} & 0.9 & 0.9527 & 0.0819/0.2611 \\
\end{tabular}
\caption{Sensitivity of the weighted angle distance to $\rho$: for each dataset and metric we report the best $\rho$ (maximizing the metric) and the variability across the sweep $\rho\in\{0.1,0.2,\dots,1.0\}$.}
\label{tab:rho_stability}
\end{table}

\begin{table}[t]
\centering
\begin{tabular}{l r r r}
\toprule
Dataset & $\rho_s(\text{ARI},\text{NMI})$ & $\rho_s(\text{ARI},\text{sil})$ & $\rho_s(\text{NMI},\text{sil})$ \\
\midrule
\texttt{splice} & 0.937 & -- & -- \\
\texttt{strseq} & 0.998 & -0.439 & -0.458 \\
\texttt{ucsc\_trf} & 0.974 & 0.136 & 0.097 \\
\end{tabular}
\caption{Spearman rank correlation $\rho_s$ between evaluation metrics across the 19 distances tested (per dataset).}
\label{tab:rank_consistency}
\end{table}

\begin{table}[t]
\centering
\begin{tabular}{l l r}
\toprule
Dataset & Family & Median pairwise time \\
\midrule
\texttt{splice} & \texttt{edit} & 2.5s \\
\texttt{splice} & \texttt{kgram\_angle} & 1.2m \\
\texttt{splice} & \texttt{js\_kgram} & 3.2m \\
\texttt{splice} & \texttt{weighted\_angle} & 19.3m \\
\addlinespace
\texttt{strseq} & \texttt{edit} & 2.1s \\
\texttt{strseq} & \texttt{kgram\_angle} & 39.0s \\
\texttt{strseq} & \texttt{js\_kgram} & 1.2m \\
\texttt{strseq} & \texttt{weighted\_angle} & 19.7m \\
\addlinespace
\texttt{ucsc\_trf} & \texttt{edit} & 25.9s \\
\texttt{ucsc\_trf} & \texttt{kgram\_angle} & 9.1m \\
\texttt{ucsc\_trf} & \texttt{js\_kgram} & 17.4m \\
\texttt{ucsc\_trf} & \texttt{weighted\_angle} & 5.26h \\
\addlinespace
\end{tabular}
\caption{Median wall-clock time to compute the full pairwise distance matrix for each family on each dataset (pure Python implementation; no caching or parallelism).}
\label{tab:runtime}
\end{table}

\begin{figure}[t]
\centering
\includegraphics[width=0.85\linewidth]{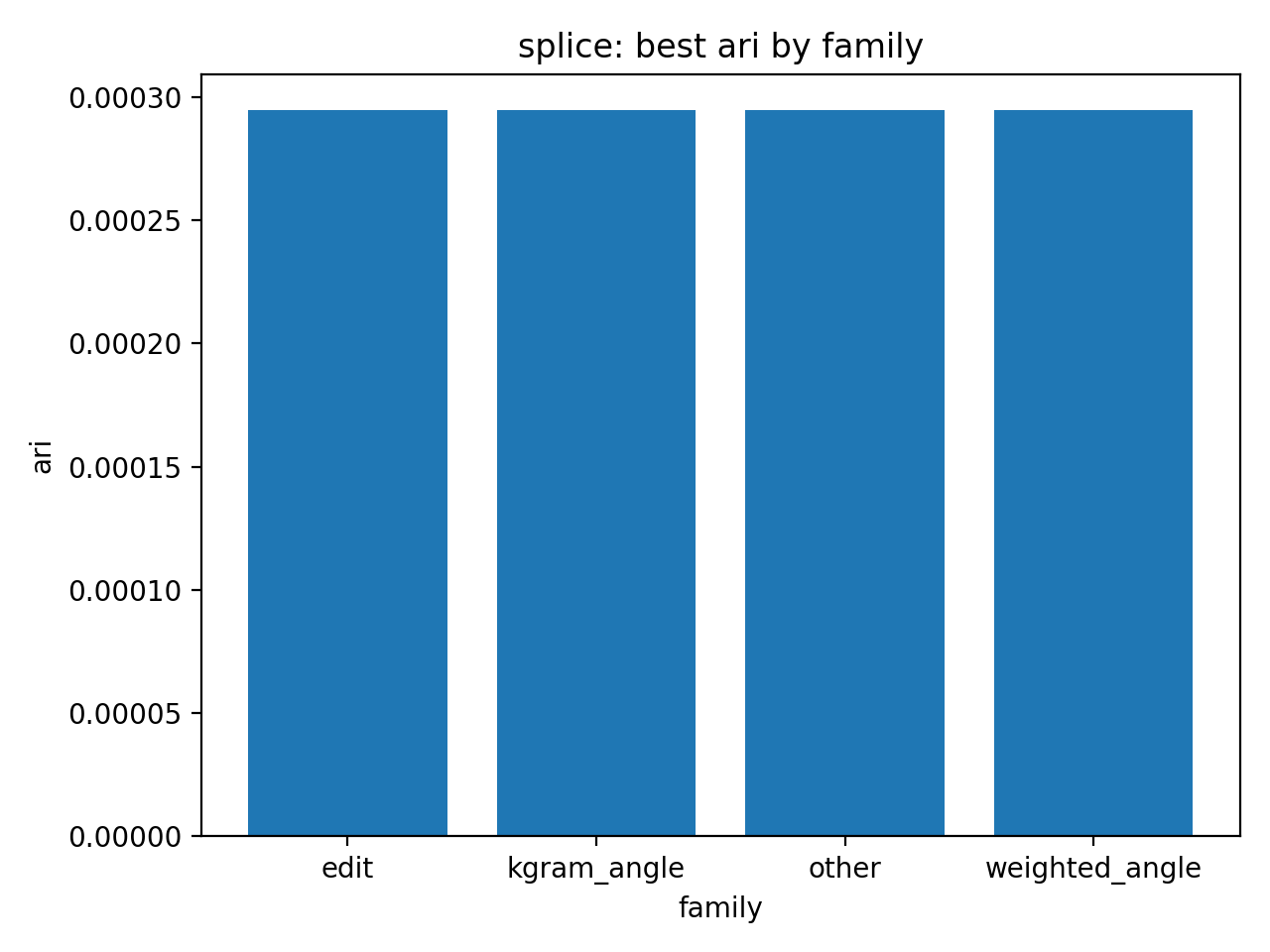}
\caption{Best ARI achieved by each distance family on \texttt{splice}.}
\label{fig:splice:best_ari_family}
\end{figure}

\begin{figure}[t]
\centering
\includegraphics[width=0.85\linewidth]{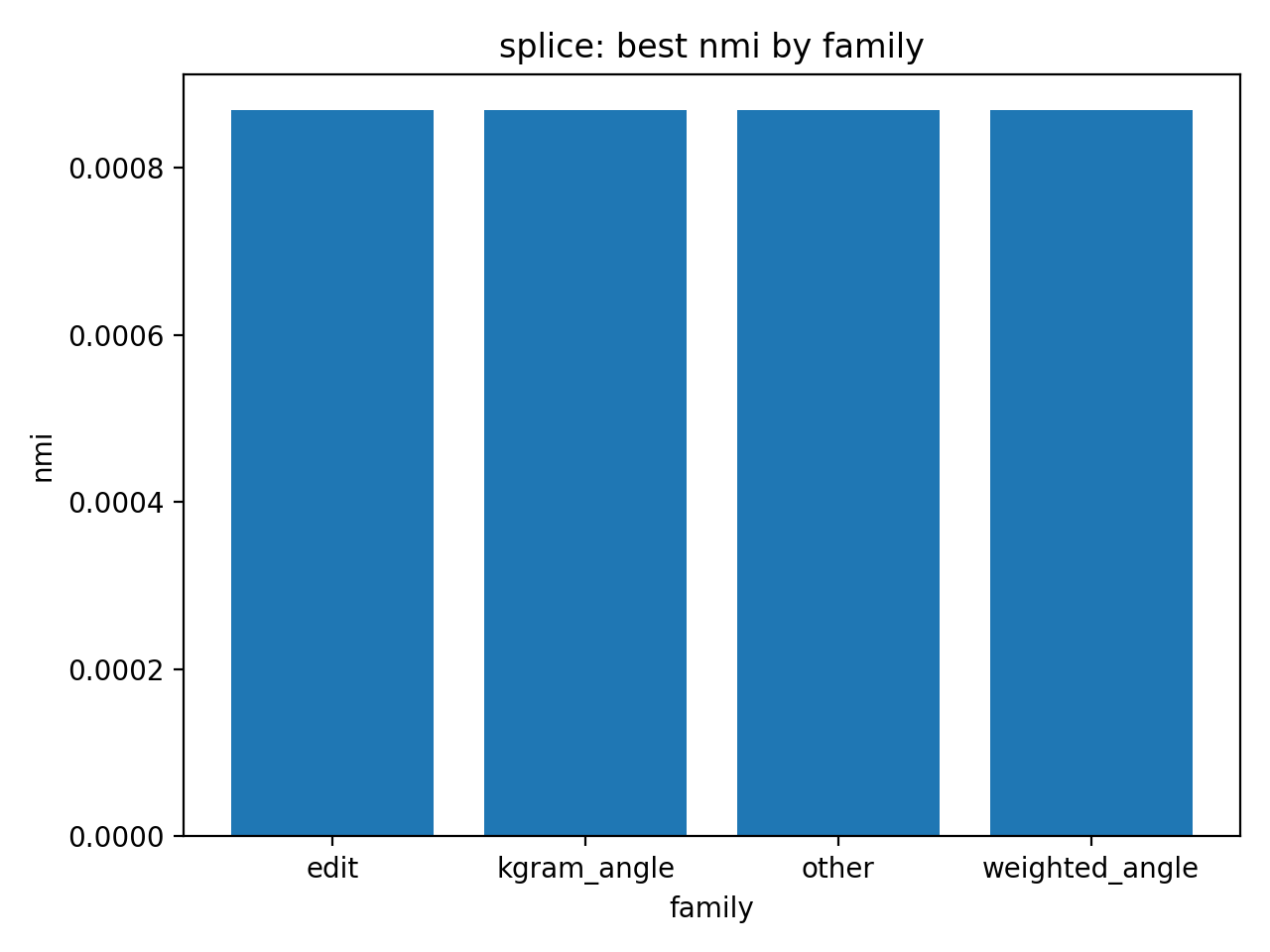}
\caption{Best NMI achieved by each distance family on \texttt{splice}.}
\label{fig:splice:best_nmi_family}
\end{figure}

\begin{figure}[t]
\centering
\includegraphics[width=0.85\linewidth]{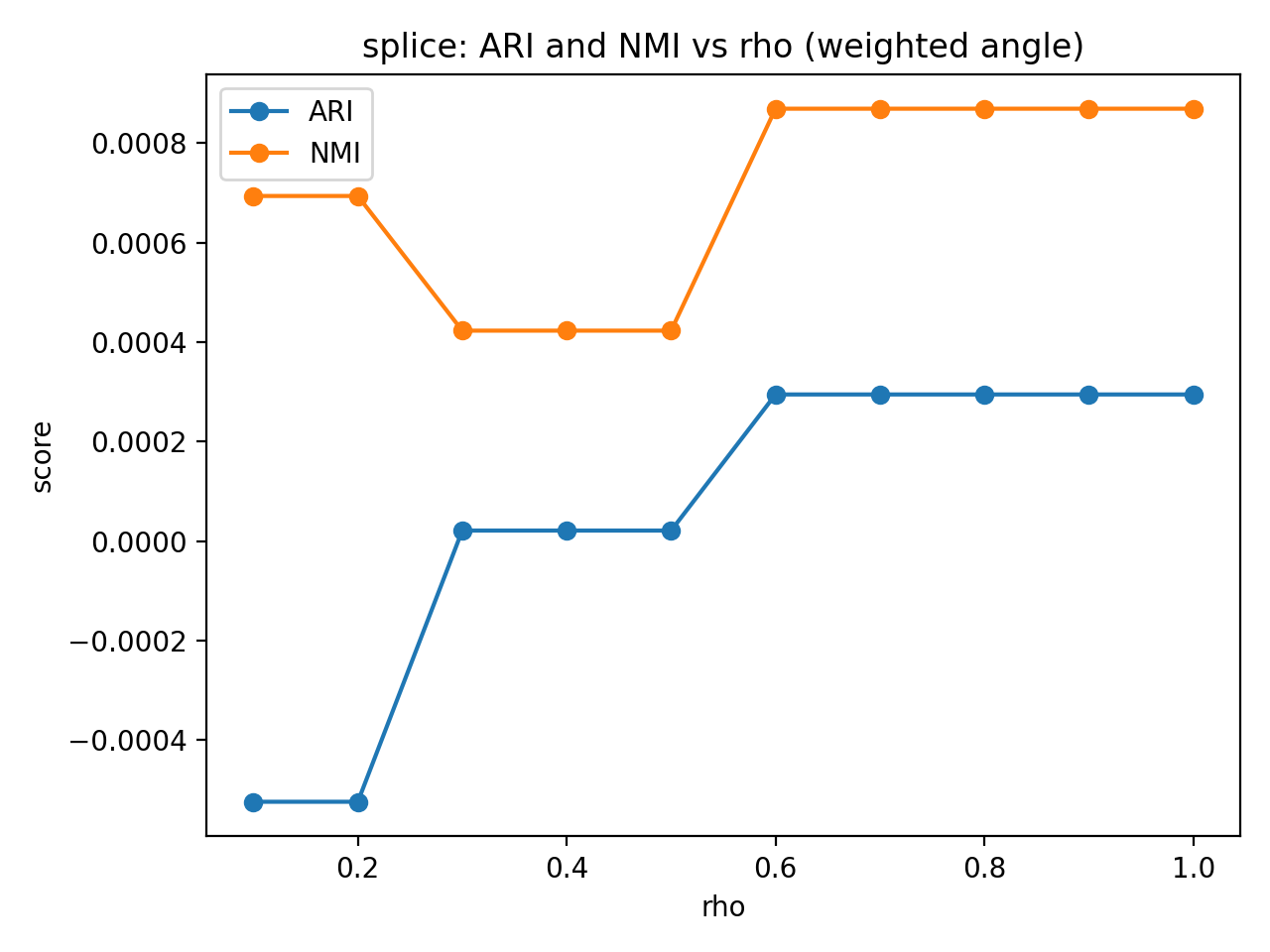}
\caption{ARI and NMI versus $\rho$ for the weighted angle distance on \texttt{splice}.}
\label{fig:splice:wa_ari_nmi_rho}
\end{figure}

\begin{figure}[t]
\centering
\includegraphics[width=0.85\linewidth]{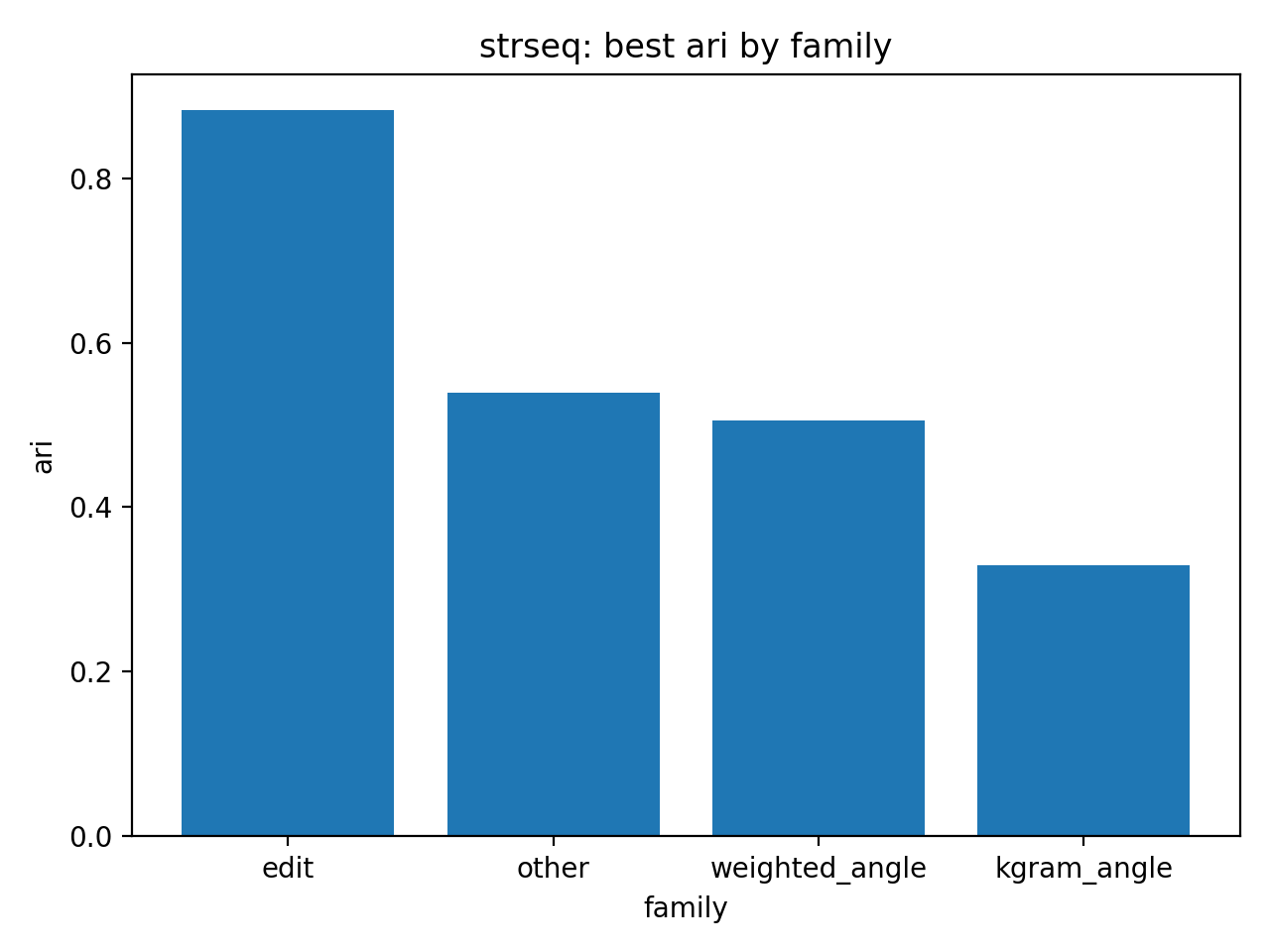}
\caption{Best ARI achieved by each distance family on \texttt{strseq}.}
\label{fig:strseq:best_ari_family}
\end{figure}

\begin{figure}[t]
\centering
\includegraphics[width=0.85\linewidth]{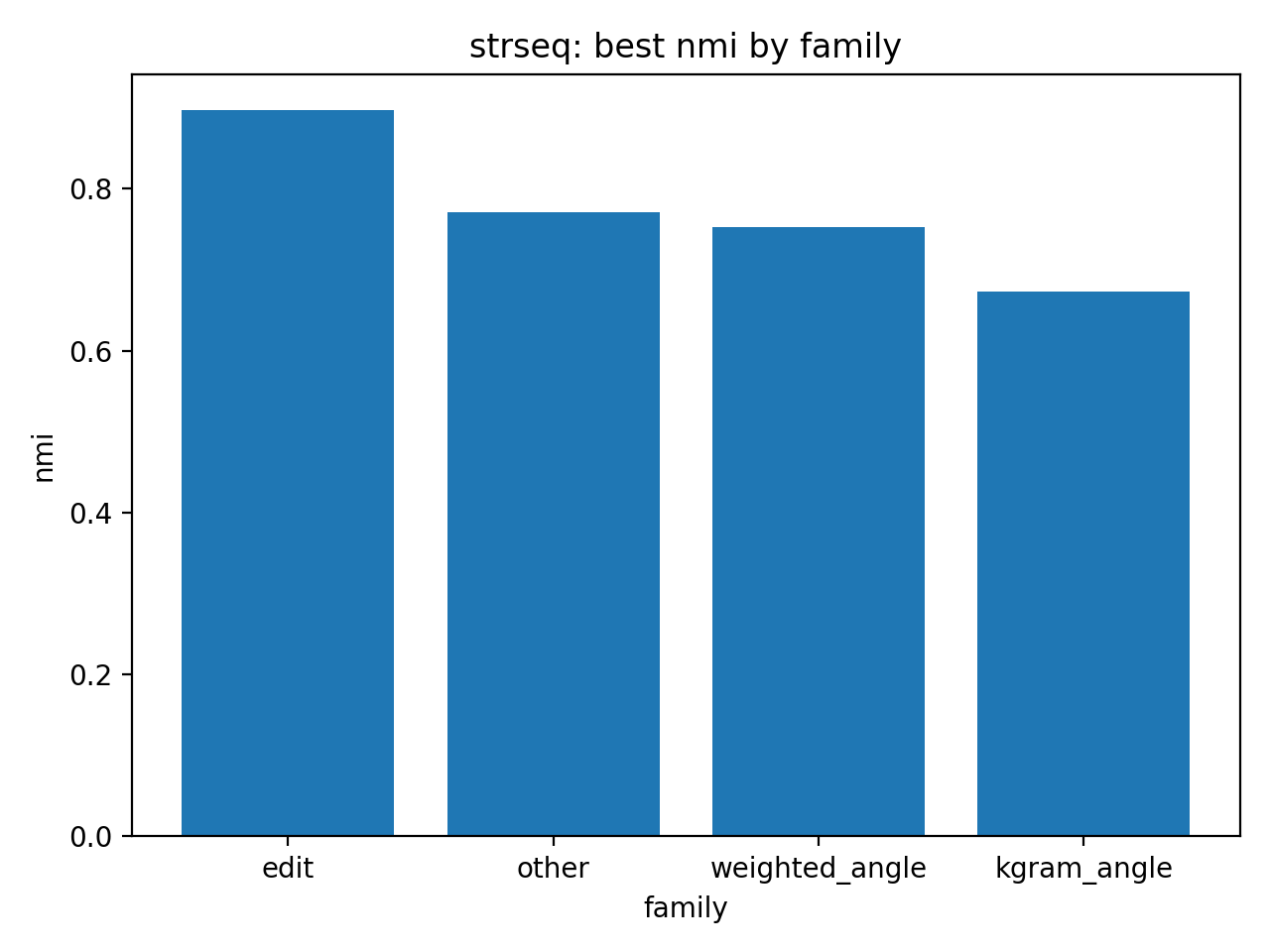}
\caption{Best NMI achieved by each distance family on \texttt{strseq}.}
\label{fig:strseq:best_nmi_family}
\end{figure}

\begin{figure}[t]
\centering
\includegraphics[width=0.85\linewidth]{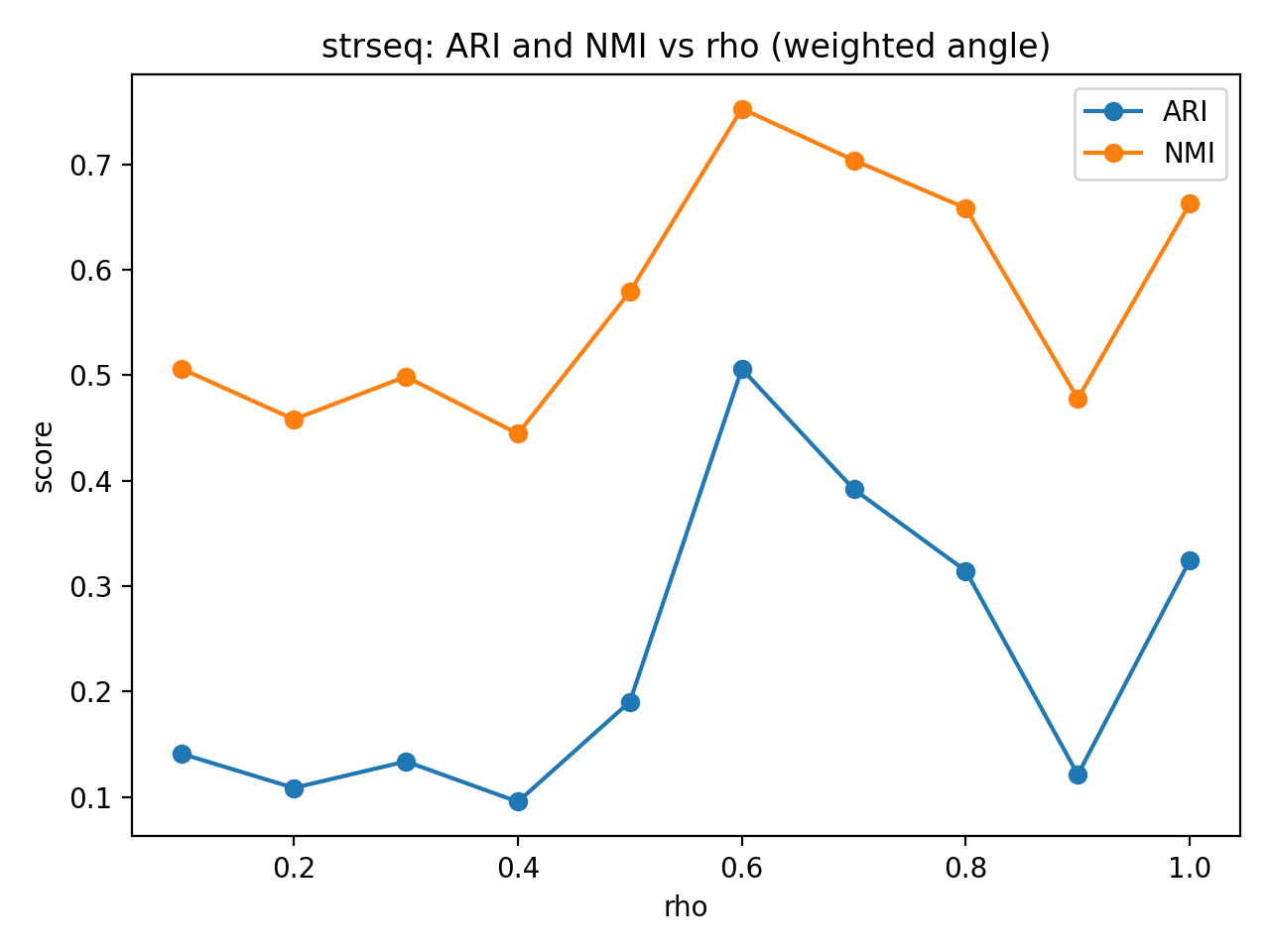}
\caption{ARI and NMI versus $\rho$ for the weighted angle distance on \texttt{strseq}.}
\label{fig:strseq:wa_ari_nmi_rho}
\end{figure}

\begin{figure}[t]
\centering
\includegraphics[width=0.85\linewidth]{figures//ucsc\_trf__best_ari_by_family.png}
\caption{Best ARI achieved by each distance family on \texttt{ucsc\_trf}.}
\label{fig:ucsctrf:best_ari_family}
\end{figure}

\begin{figure}[t]
\centering
\includegraphics[width=0.85\linewidth]{figures//ucsc\_trf__best_nmi_by_family.png}
\caption{Best NMI achieved by each distance family on \texttt{ucsc\_trf}.}
\label{fig:ucsctrf:best_nmi_family}
\end{figure}

\begin{figure}[t]
\centering
\includegraphics[width=0.85\linewidth]{figures//ucsc\_trf__weighted_angle__ari_nmi__vs_rho.png}
\caption{ARI and NMI versus $\rho$ for the weighted angle distance on \texttt{ucsc\_trf}.}
\label{fig:ucsctrf:wa_ari_nmi_rho}
\end{figure}

\begin{figure}[t]
\centering
\includegraphics[width=0.85\linewidth]{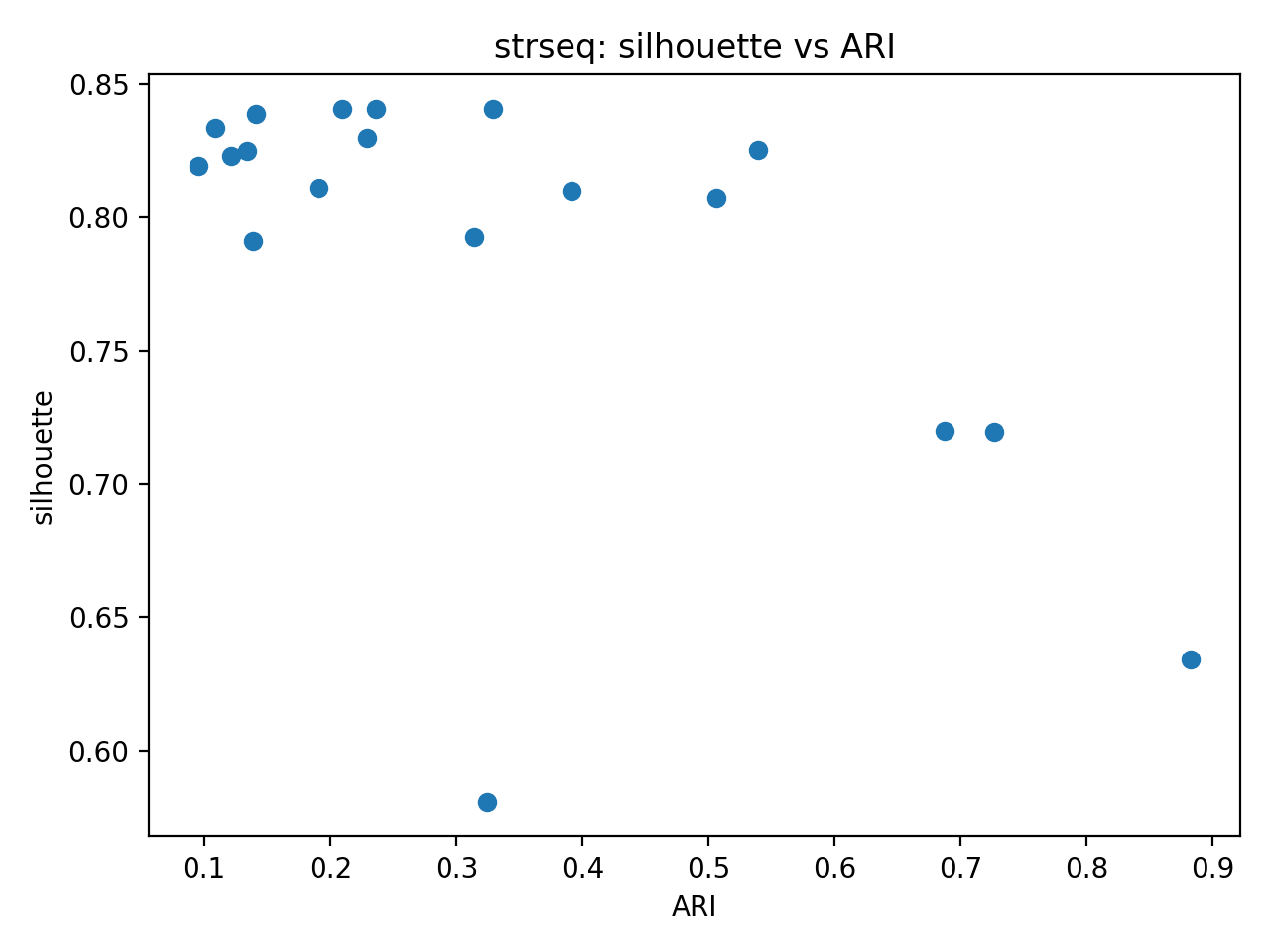}
\caption{Silhouette versus ARI across all tested distances on \texttt{strseq}.}
\label{fig:strseq:sil_vs_ari}
\end{figure}

\begin{figure}[t]
\centering
\includegraphics[width=0.85\linewidth]{figures//ucsc\_trf__silhouette__vs_ari.png}
\caption{Silhouette versus ARI across all tested distances on \texttt{ucsc\_trf}.}
\label{fig:ucsctrf:sil_vs_ari}
\end{figure}

\begin{figure}[t]
\centering
\includegraphics[width=0.85\linewidth]{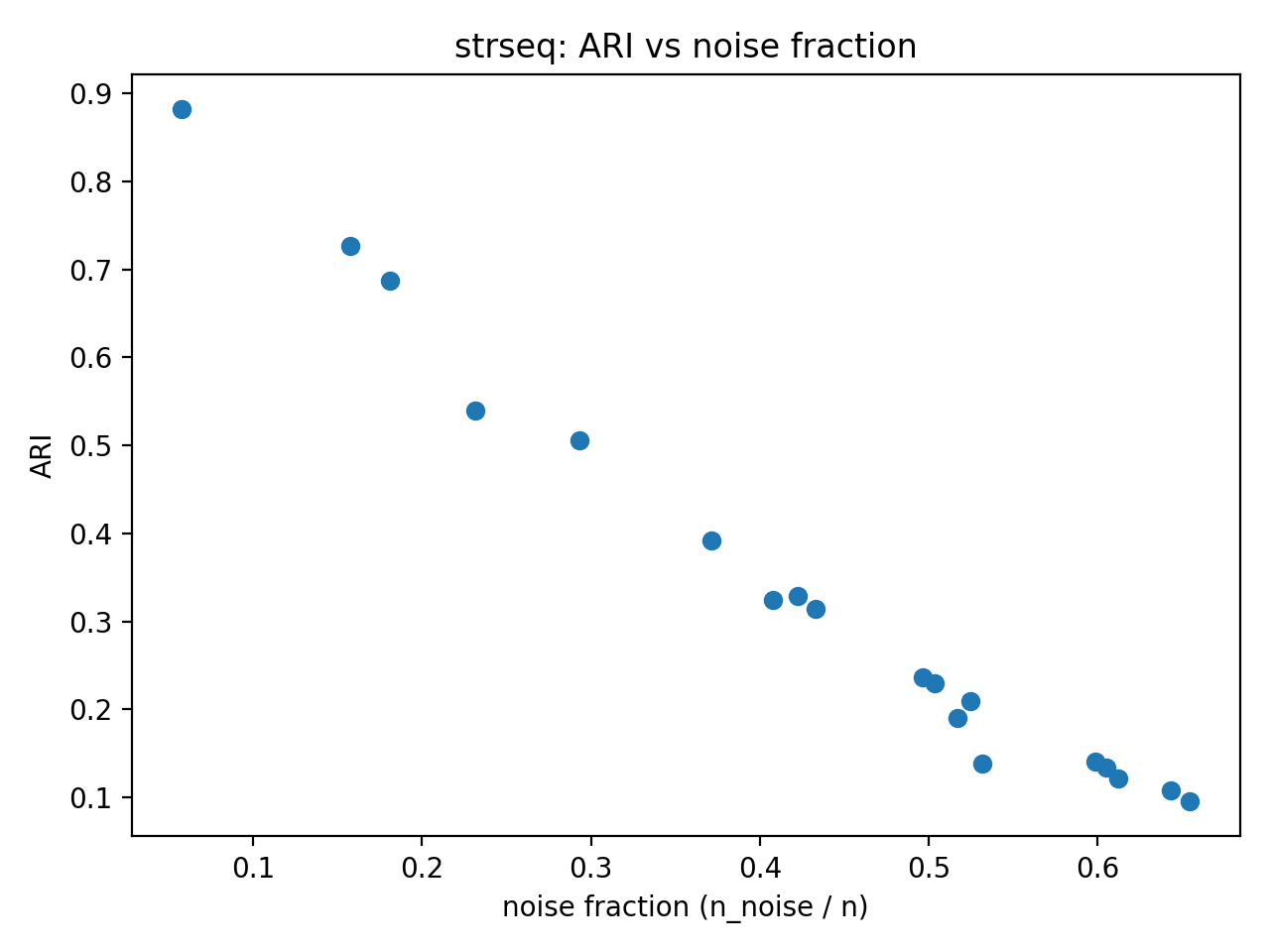}
\caption{ARI versus DBSCAN noise fraction across all tested distances on \texttt{strseq}.}
\label{fig:strseq:ari_vs_noise}
\end{figure}

\begin{figure}[t]
\centering
\includegraphics[width=0.85\linewidth]{figures//ucsc\_trf__ari_vs_noise_frac.png}
\caption{ARI versus DBSCAN noise fraction across all tested distances on \texttt{ucsc\_trf}.}
\label{fig:ucsctrf:ari_vs_noise}
\end{figure}

\appendix

\bibliography{weightedanglebib}{}
\bibliographystyle{amsplain}

\end{document}

%% file: 1_Alphabet_Commands.tex
\newcommand{\bbN}{\mathbb{N}}

\newcommand{\bbR}{\mathbb{R}}

\newcommand{\bbZ}{\mathbb{Z}}

%% file: 2_Math_Commands.tex
\newcommand{\parent}[1]{\left( #1 \right)}

\newcommand{\inv}{^{-1}}

\newcommand{\abs}[1]{\left|#1\right|}
\newcommand{\norm}[1]{\left\lVert #1 \right\rVert}
\newcommand{\Zmod}[1]{\faktor{\bbZ}{n \bbZ}}
\newcommand{\floor}[1]{\left\lfloor #1 \right\rfloor}

\input cyracc.def